\documentclass[11pt]{amsart}
\usepackage{amsmath,amssymb,latexsym,dsfont}
\usepackage[left=2.7cm,right=2.7cm,top=2.7cm,bottom=2.7cm]{geometry}

\usepackage{amsmath,amsfonts,amssymb,epsfig, textcomp}
\usepackage{graphicx}
\usepackage{hyperref}
\usepackage{cases}
\usepackage{color}

\newtheorem{theorem}{Theorem}[section]
\newtheorem{lemma}{Lemma}[section]

\newtheorem{proposition}{Proposition}[section]
\newtheorem{corollary}{Corollary}[section]
\newtheorem{remark}{Remark}[section]

\newtheorem{definition}{Definition}[section]
\numberwithin{equation}{section}

      \newcommand{\R}{{\mathbb{R}}}

      \newcommand{\curl}{\operatorname{curl}}
      \newcommand{\dive}{\operatorname{div}}
      \newcommand{\N}{\mathbb{N}}
      \newcommand{\ext}{\operatorname{ext}}
      \newcommand{\inte}{\operatorname{int}}
      \newcommand{\loc}{\operatorname{loc}}
      
      \newcommand{\eps}{\varepsilon}
      \newcommand{\mR}{\mathbb{R}}

      \newcommand{\EE}{\mathcal{E}}  
      \newcommand{\HH}{\mathcal{H}}
       \newcommand{\hEE}{\hat{\mathcal{E}}}  
      \newcommand{\hHH}{\hat{\mathcal{H}}}
      \newcommand{\EEE}{\mathbb{E}}  
      \newcommand{\HHH}{\mathbb{H}}
      \newcommand{\JJJ}{\mathbb{J}}

      \newcommand{\supp}{\operatorname{supp}}
      \newcommand{\dsp}{\displaystyle}
      \newcommand{\bE}{\mathbf E}
\newcommand{\bH}{\mathbf H}
      \newcommand{\tbE}{\widetilde{\mathbf E}}
\newcommand{\tbH}{\widetilde{\mathbf H}}

      \makeatletter
      \def\@setcopyright{}
      \def\serieslogo@{}
      \makeatother

\begin{document}

%



   \author{}
   \address{}
   \email{}


   \author{}

   \address{}

   \curraddr{}

   \email{}
   

   \title[Cloaking for Maxwell's equations]{Approximate cloaking for time-dependent Maxwell equations via transformation optics}

      \author[H.-M. Nguyen]{Hoai-Minh Nguyen}
\author[L. Tran]{Loc Tran}

\address[H.-M. Nguyen]{Department of Mathematics, EPFL SB CAMA, Station 8,  \newline\indent
	 CH-1015 Lausanne, Switzerland.}
\email{hoai-minh.nguyen@epfl.ch}

\address[L. Tran]{Department of Mathematics, 
	EPFL SB CAMA Station 8, 
	\newline\indent CH-1015 Lausanne, Switzerland.}
\email{loc.tran@epfl.ch}
   
\begin{abstract} We study approximate cloaking using transformation optics for electromagnetic waves in the time domain. Our approach is based on estimates of the degree of visibility in the frequency domain for all frequencies in which the frequency dependence is explicit. The difficulty and the novelty analysis parts are  in the low and high frequency regimes. To this end, we implement a variational technique in the low frequency domain, and  multiplier and duality techniques in the high frequency domain.  Our approach is inspired by
the work of Nguyen and Vogelius  on the wave equation. 
 
\end{abstract}

   \dedicatory{}


   \maketitle



\noindent{\bf Key words.} cloaking, transformation optics, Maxwell's equations, radiation condition. 

\medskip 
\noindent{\bf AMS subject classification.}  35A15, 35B40, 35F10, 78A40, 78M30. 

\section{Introduction and  statements of results}

Cloaking via transformation optics was introduced by Pendry, Schurig, and Smith \cite{Pen} for the Maxwell system and by Leonhardt \cite{Leo} in the geometric optics setting. 
The idea is to use the invariance of Maxwell equations under a change of variables. 
They used a singular change of variables that blows up a point into a
cloaked region. The same transformation was used  by Greenleaf, Lassas,
and Uhlmann in an inverse context \cite{Green}.  However, the singular nature of the cloaks presents various difficulties in practice as well as in theory: (1) they are hard to fabricate and (2) in certain cases, the correct definition (and therefore the properties) of the corresponding electromagnetic fields is an issue. To avoid using
the singular structure, various regularized schemes have been proposed. One of them  was suggested by Kohn, Shen, Vogelius, and Weinstein  \cite{Kohn} in which they  used a  transformation which blows up a small ball of radius $\rho$ instead of a point 
into the cloaked region.  Other, related regularizations schemes have also been proposed \cite{RYNQ,GKLU07}. 
 It is worth mentioning that there are other techniques for cloaking,  some of which  use negative index materials such as cloaking using complementary media, see,  e.g.,    \cite{LaiChenZhangChanComplementary, Ng-Negative-Cloaking}  and cloaking via localized resonance, see,  e.g., \cite{Ng-CALR-O} (see also \cite{NicoroviciMcPhedranMilton94, MiltonNicorovici, Ng-CALR}).

Approximate cloaking  using transformation optics for  the acoustic setting has been investigated in the last fifteen years. In the frequency domain, if an appropriate or a fixed lossy layer (damping layer) is implemented between the transformation cloak and the cloaked region, then cloaking is achieved, and the degree of visibility is of the order $\rho$ in three dimensions and $1/ |\ln \rho|$ in two dimensions, see  \cite{Kohn, HM3} respectively. Without such a lossy layer, the phenomena are more complex and have been investigated in more depth \cite{HM2}. In this setting, there are two distinct situations: resonant and non-resonant. In the non-resonant case, cloaking is achieved with the same degree of visibility; however, the field inside the cloaked region might depend on the field outside (cloaking vs shielding). In the resonant case, the energy inside the cloaked region can blow up, and  cloaking might not be achieved.  Different cloaking aspects related to the Helmholtz equation such as zero frequency context and the enhancement,  have been studied \cite{Kohn, NgV1,  Ammari13-1, GV, HV} and references therein. There are much less rigorous works in the time domain.  Cloaking using transformation optics for the wave equation was established in which  a lossy layer is also used \cite{HM6}, and in which  the dispersion of the transformation cloak using the Drude-Lorentz model is accounted and a fixed lossy layer is used \cite{HM7}, in this direction.   

In the electromagnetic time harmonic context, the situation on one hand shares some common features with the scalar case and on the other hand has some distinct figures, see \cite{HM9}.  In the non-resonant electromagnetic case, without sources inside the cloaked region, it is shown that cloaking is achieved and  the degree of visibility is of the order $\rho^3$. In the resonant electromagnetic case, in contrast to the scalar case, cloaking is always achieved even if the energy inside the cloaked region might blow up.  Moreover,  the degree of visibility varies between the non-resonant and resonant cases. Other works on cloaking for the Maxwell equations in the time harmonic regime can be found in \cite{GKLU07-1, Weder1, Weder2, Ammari13, DLU, Lassas} and references therein.

This paper is devoted to cloaking using transformation optics for the Maxwell equations in the {\it time} domain. We  use the regularization transformation instead of the singular one for the starting point,  which is necessary for viewing previous results in the time harmonic regime. Concerning the analysis,  we first transform the Maxwell equations in the time domain  into a family of the Maxwell equations in the time harmonic regime by taking the Fourier transform of the solutions with respect to time. After obtaining appropriate estimates on the  near invisibility  of the Maxwell equations in the time harmonic regime, we simply invert the Fourier transform. This idea has its roots in the work of Nguyen and Vogelius \cite{HM6} (see also \cite{HM7}) in the  context of acoustic cloaking and was used to study  impedance boundary conditions in the time domain  \cite{HM-VL} and cloaking for the heat equation \cite{Minh-Tu}. To implement this idea, the heart of the matter is to obtain the degree of visibility in which  the dependence on frequency is {\it explicit} and well controlled.  The analysis involves a variational method, a multiplier technique,  and a duality argument in different ranges of frequency. 
An intriguing fact about the Maxwell equations in the time harmonic regime worth mentioned is that the multiplier technique does not fit well  for the purposes of  cloaking in the  very high frequency regime, and a duality argument is involved instead. 
Another key technical point is the proof of the radiating condition for the Fourier transform in time of the weak solutions of the general Maxwell equations, a fact which  is interesting in itself. 
Note that after a change of variables, the study of the cloaking effect can be derived from the study of the effect of a small inclusion which is known when the coefficients inside the small inclusions are fixed (or has a finite range), generally for a fixed frequency, see, e.g., \cite{AVV, VogeliusVolkov}.
Nevertheless, the situation in the  context of cloaking is non-standard since the coefficients inside the small inclusion blow up as the diameter goes to 0.

Let us now describe the problem in more detail. For simplicity, we suppose that the cloaking device occupies
the annular region $B_2 \setminus B_{1/2}$ and the cloaked region is the ball $B_{1/2}$ in $\mR^3$ in which the permittivity and the permeability are given by two $3\times 3$ matrices $\eps_O, \mu_O$, respectively. In this paper, for $r>0$,  we denote $B_r$ as the ball centered at the origin and of radius $r$. Throughout this paper, we assume that, in $B_{1/2}$,  
\begin{equation}\label{m-1}
\eps_O, \,  \mu_O \mbox{ are real, symmetric}, 
\end{equation}
and uniformly elliptic, i.e., 
\begin{equation}\label{m-2}
\frac{1}{\Lambda} |\xi|^2 \le \langle \eps_O(x)\xi, \xi \rangle, \langle \mu_O (x)\xi, \xi\rangle \le \Lambda|\xi|^2 \quad  \forall \,  \xi \in \mR^3,  
\end{equation}
for a.e. $x\in B_{1/2}$ and for some $\Lambda  \ge 1$. We also  assume  $\eps_O, \mu_O$ are piecewise $C^1$  to ensure the uniqueness of solutions via the unique continuation principle (see \cite{TNguyen, BC}, and also \cite{Protter}). 

Let $\rho\in (0, 1)$ and let $F_{\rho}: \R^3 \rightarrow \R^3$ be defined by
\begin{align}F_{\rho} (x)= \left\{ \begin{array}{cl} x &\text{ in } \R^3\backslash B_2, \\[6pt] 
\dsp  \left( \frac{2 - 2\rho}{2-\rho} + \frac{|x|}{2-\rho} \right)\frac{x}{|x|} &\text{ in } B_2\backslash B_{\rho},\\[6pt]
\dsp \frac{x}{\rho} &\text{ in } B_{\rho}.\end{array} \right. \end{align}
The cloaking device in $B_2 \setminus B_{1/2}$ constructed via the transformation optics technique is characterized by the triple of permittivity, permeability, and conductivity and  contains two layers. The first one in $B_2 \setminus B_{1}$ that comes from the transformation technique using the map $F_\rho$ is 
$$
({F_{\rho}}_*I, {F_{\rho}}_*I, 0 )  
$$
and the second one in $B_{1} \setminus B_{1/2}$, which is a fixed lossy layer, is 
$$
(I, I, 1). 
$$
Here and in what follows, for a diffeomorphism $F$ and a matrix-valued function $A$, one denotes
\begin{equation}\label{A*}
F_*A : = \frac{DF A DF^{T} }{|\det DF| }\circ F^{-1}. 
\end{equation}

\begin{remark} \rm Different fixed lossy layer can be used. However, to  simplify  the notations and to avoid several unnecessary technical points, the triple $(I, I, 1)$ is considered. 
\end{remark}

Assume that the medium is homogeneous outside the cloaking device and the cloaked region. In the presence of the cloaked object and the cloaking device, the medium in the whole space $\R^3$ is described by the triple $(\eps_c, \mu_c, \sigma_c)$ given by 
\begin{equation}\label{medium-cloak}
(\eps_c, \mu_c, \sigma_{c}) =  
\left\{\begin{array}{cl}
(I, I , 0) &\mbox{ in } \R^3\setminus B_2,\\[6pt]
({F_{\rho}}_*I, {F_{\rho}}_*I, 0) & \mbox{ in } B_2\setminus B_{1},\\[6pt]
(I , I, 1) & \mbox{ in } B_1\setminus B_{1/2}, \\[6pt]
(\eps_O, \mu_O, 0) & \mbox{ in } B_{1/2}.
\end{array}\right. 
\end{equation}


Let $\mathcal J$ represent a charge density. We assume that
\begin{equation}\label{j-1}
\mathcal J\in L^1([0, \infty);[L^2(\R^3)]^3) \mbox{ with } \operatorname{supp} \mathcal J \subset [0, T]\times (B_{R_0} \setminus B_2), \mbox{ for some } T> 0, R_0>2, 
\end{equation}
and 
\begin{equation}\label{j-2}
\dive \mathcal J = 0 \mbox{ in } \mR_+ \times \mR^3. 
\end{equation}
With the cloaking device  and the cloaked object, the electromagnetic wave generated by $\mathcal J$ with zero data at  time $0$ is the unique weak solution $(\EE_c, \HH_c) \in L^{\infty}_{\loc}([0, \infty), [L^2(\R^3)]^6)$ to the system
\begin{equation}
\label{equ:wcloak}
\begin{cases}
\dsp  \eps_c \frac{\partial \EE_c}{\partial t} =  \nabla \times \HH_c -  \mathcal J -  \sigma_c \EE_c &\text{ in } (0, +\infty)\times\mathbb{R}^3,\\[6pt]
\dsp  \mu_c \frac{\partial \HH_c}{\partial t}  = - \nabla \times \EE_c   &\text{ in } (0, +\infty)\times\mathbb{R}^3,\\[6pt]
\EE_c(0, \cdot) = \HH_c(0,\cdot ) = 0 & \text{ in } \R^3.
\end{cases}\end{equation}
In the homogeneous space, the field generated by $\mathcal J$ with zero data at time $0$ is the unique weak solution $(\EE, \HH)\in L^{\infty}_{\loc}([0, \infty), [L^2(\R^3)]^6)$ to the system 
\begin{align}
\label{equ:ncloak}
\begin{cases}
\dsp \frac{\partial \EE}{\partial t}  =  \nabla \times \HH - \mathcal J  &\text{ in } (0, +\infty)\times\mathbb{R}^3,\\[6pt]
\dsp  \frac{ \partial \HH}{\partial t} = - \nabla \times \EE  &\text{ in } (0, +\infty)\times\mathbb{R}^3,\\[6pt]
\EE(0,\cdot ) = \HH(0, \cdot ) = 0 & \text{ in } \R^3.
\end{cases}
\end{align}

The meaning of weak solutions, in a slightly more general context,  is as follows. 

\begin{definition}\label{def} Let  $\eps,  \, \mu,  \,  \in [L^{\infty}(\R^3)]^{3\times 3}$, $\sigma_m, \,  \sigma_e \in L^{\infty}(\R^3)$ be such that $\eps$ and $\mu$ are real,  symmetric, and uniformly elliptic in $\mR^3$, and  $\sigma_m$ and $\sigma_e$ are real and nonnegative  in $\mR^3$, 
and  let $f_e, f_m \in L^1_{\loc}([0, \infty); [L^2(\R^3)]^3)$. A pair  $(\EE, \HH)\in L^{\infty}_{\loc}([0, \infty), [L^2(\R^3)]^6)$ is called a weak solution of 
\begin{equation}\label{equ}
\begin{cases}
\dsp  \eps \frac{\partial \EE}{\partial t} = \nabla\times \HH -  \sigma_e\EE + f_m &\text{ in } (0, +\infty)\times \R^3,\\[6pt]
\dsp\mu \frac{\partial \HH}{\partial t} = -  \nabla\times \EE  -\sigma_m\HH  + f_e&\text{ in } (0, +\infty)\times \R^3,\\[6pt]
\dsp \EE(0, \cdot) = 0; \HH(0, \cdot) = 0 & \text{ in } \R^3, 
\end{cases}
\end{equation}
if  
\begin{equation}\label{def-e1}
\begin{cases}
\dsp \frac{d}{dt}\langle \eps\EE(t, .), E\rangle + \langle \sigma_e \EE(t, .), E\rangle - \langle \HH(t, .), \nabla\times E\rangle = \langle f_m(t, .), E\rangle, \\[8pt]
\dsp \frac{d}{dt}\langle \mu\HH(t, .), H\rangle + \langle \sigma_m \HH(t, .), H\rangle + \langle \EE(t, .), \nabla\times H\rangle = \langle f_e(t, .), H\rangle,
\end{cases} \quad \mbox{ for } t > 0, 
\end{equation}
for all $(E, H)\in [H(\curl, \mR^3)]^2$, and 
\begin{equation}\label{def-e2}
\EE(0, .) = \HH(0, .) = 0 \mbox{ in } \R^3.
\end{equation}
\end{definition}

Some comments on Definition~\ref{def} are in order. System \eqref{def-e1} is understood in the distributional sense. Initial condition   \eqref{def-e2}  is understood as 
\begin{equation}\label{trace-sense}
\langle \eps \EE(0, .), E \rangle  = \langle \mu \HH(0, .), H \rangle = 0  \quad \mbox{ for all } (E, H)\in [H(\curl, \mR^3)]^2. 
\end{equation}
From \eqref{def-e1}, one can check that  
\begin{equation*}
\langle \eps \EE(t, .), E \rangle, \langle \mu \HH(t, .), H \rangle \in W^{1, 1}_{\loc}([0, + \infty)). 
\end{equation*}
This in turn ensures the trace sense  in \eqref{trace-sense}. 

\medskip 
Concerning the well-posedness of \eqref{equ}, we have, see,  e.g., \cite[Theorem 3.1]{HM8}, 
  
\begin{proposition}\label{prop-wellposed} Let $f_e, f_m \in L^1_{\loc}([0, \infty); [L^2(\R^3)]^3)$. There exists a unique weak solution $(\EE, \HH) \in L^{\infty}_{\loc}([0, \infty), [L^2(\R^3)]^6)$ of \eqref{equ}. Moreover, for $T> 0$, the following estimate holds
	\begin{equation}\label{es}
	\int_{\R^3}|\EE (t, x) |^2 + |\HH (t, x) |^2 dx \leq C\left(\int\limits_{0}^t\Big\|\big(f_e(s, .), f_m(s, .)\big)\Big\|_{L^2(\R^3)} ds\right)^2 \quad \mbox{ for } t\in [0,T],
	\end{equation}
	for some positive constant $C$ depending only on the ellipticity of  $\eps$ and $\mu$. 
	
\end{proposition}

\begin{remark} \rm We emphasize here that the constant $C$ in Proposition~\ref{prop-wellposed} is independent of $T$. This fact is later used in the proof of  the radiating condition.  In \cite{HM8}, the authors considered dispersive materials and also dealt with Maxwell equations which are non-local in time. 
\end{remark}

We are ready to state the main result of the paper that is proved in Section~\ref{sc-proof-main}.
\begin{theorem}\label{thm-main-1}
	Let $\rho \in (0, 1)$, $T> 0$ and  let $(\EE_c, \HH_c), (\EE, \HH)\in L_{\loc}^{\infty}([0, \infty), [L^2(\R^3)]^6)$ be the unique solutions to systems \eqref{equ:wcloak} and \eqref{equ:ncloak}, respectively. Assume \eqref{j-1} and \eqref{j-2}. Then, for $K\subset\subset \R^3\backslash \bar B_1$, 
	\begin{equation}\label{main-1}
	\| (\EE_c, \HH_c) - (\EE, \HH)\|_{L^{\infty}((0, T); L^2(K))}\leq C T \rho^3\|\mathcal J\|_{H^{11}((0, \infty); [L^2(\R^3)]^3)},
	\end{equation}	
for some positive constant $C$ depending only on $K$, $R_0$.
\end{theorem}

\begin{remark} \rm Assertion \eqref{main-1} is optimal since it gives the same degree of visibility as in the  frequency domain in \cite{HM9} where  the optimality is established. 
\end{remark}

\begin{remark} \rm Estimate \eqref{main-1} requires that  $\mathcal J$ is regular. The condition of the  regularity of $\mathcal J$ is not optimal,  and this optimality would be studied elsewhere. 
\end{remark}

Our approach is inspired  by the  work of  Nguyen and Vogelius  \cite{HM6} (see also \cite{HM7, HM-VL}), where they studied approximate cloaking for the acoustic setting in the time domain. The main idea can be briefly described  as follows. 
We first transform the time-dependent Maxwell systems into a family of the time-harmonic Maxwell systems by taking the Fourier transform of the solutions with respect to time. After obtaining the appropriate degree of near invisibility for the Maxwell equations in the time harmonic regime, where the frequency dependence  is {\it explicit}, we simply invert the Fourier transform.
The analysis in the frequency domain $\omega$ (in Section~\ref{sc-freq}) can be divided into three steps that deal with frequencies in 
low and moderate ($0 < \omega < 1$), moderate and high ($1 < \omega < 1/ \rho$), high and very high  ($\omega > 1/\rho$) regimes. The analysis in the low and moderate frequency regime (in Section~\ref{sc-low}) is based on a variational approach. In comparison with \cite{HM9}, one needs to additionally derive an estimate for small frequency in which the frequency dependence  is explicit. In the moderate and high frequency regime, to obtain appropriate estimates, 
we use the multiplier technique and  the test functions are inspired from the scalar case due to Morawetz (see \cite{MorawetzLudwig}). The analysis in the moderate and high frequency regime is given in Section~\ref{sc-av}.   There is a significant difference between the scalar case and the Maxwell vectorial case. It is known in the scalar case that one  can control the normal derivative of a solution to the exterior Helmholtz equation in homogeneous medium by its value on the boundary of a convex, bounded subset of $\mR^3$. However,  in contrast with the scalar case, one cannot either use tangential components of the electromagnetic fields to control the normal component in the same Sobolev norms and conversely.  
This fact can be seen from the explicit solutions outside a unit ball of Maxwell equations (see, e.g., \cite[Theorem 2.50]{Kirsch}). This is the reason for which we cannot use the multiplier technique in the  very high frequency regime and  again reveals the distinct structure of Maxwell equations in the time harmonic regime as compared to the Helmholtz equations.   The analysis in the high and very high frequency regime in Section~\ref{sc-high} is based on the duality method inspired from \cite{Lions}. The proof of Theorem~\ref{thm-main-1} based on the frequency analysis is given in Section~\ref{sc-proof-main}. A key technical point required for  the analysis in the frequency domain  is the establishment of  the radiation condition for the Fourier transform with respect to time of the solutions of Maxwell equations. The rigorous proof on the radiation condition  in a general setting is new to our knowledge and is interesting in itself.

The paper is organized as follows. Section \ref{sc-freq} is devoted to the estimates for Maxwell's equations in frequency domain. Section \ref{sc-proof-main} gives the proof of Theorem \ref{thm-main-1}. The assertion on the  radiation condition is also stated and proved there.

\section{Frequency analysis}\label{sc-freq}

In this section, we provide estimates to assess the degree of visibility in the frequency domain. We first recall some notations. Let $U$ be a smooth open subset of $\mR^3$. We  denote 
\[
 H(\curl, U) := \Big\{\phi \in [L^2(U)]^3: \nabla \times  \phi \in [L^2(U)]^3 \Big\}, 
 \]
\[
 H(\dive, U) := \Big\{\phi \in [L^2(U)]^3: \dive \phi \in L^2(U) \Big\}. 
 \]
We also use the notations $H_{\loc}(\curl, U)$ and $H_{\loc}(\dive, U)$ with the usual convention.

Given $\JJJ \in [L^2(\mR^3)]^3$ with compact support, let $(\EEE, \HHH) \in [H_{\loc}(\curl, \mR^3)]^2$ and $(\EEE_\rho, \HHH_\rho) \in  [H_{\loc}(\curl, \mR^3)]^2$ $(\rho  >0)$ be the corresponding unique radiating solutions of the following systems 
\begin{equation}\label{Sys-F-1}
\begin{cases}
 	\nabla \times \EEE = i \omega \HHH  &\mbox{ in } \R^3,\\[6pt]
 	\nabla\times  \HHH= -i\omega  \EEE  + \JJJ & \mbox{ in } \R^3, 
 	\end{cases}
\end{equation}
and 
\begin{equation}\label{Sys-F-1-rho}
\begin{cases}
 	\nabla\times \EEE_{\rho} = i \omega \mu_{\rho}  \HHH_{\rho} &\mbox{ in } \R^3,\\[6pt]
 	\nabla\times \HHH_{\rho} = - i \omega \eps_{\rho} \EEE_{\rho} +\sigma_{\rho} \EEE_{\rho}+ \JJJ & \mbox{ in } \R^3. 
 	\end{cases}
\end{equation}
Here, for $\rho > 0$,  
\begin{equation}\label{eps-mu-rho}
	(\eps_{\rho}, \mu_{\rho}, \sigma_{\rho}) = 
	\left\{\begin{array}{cl}
	(I, I, 0) &\mbox{ in } \R^3\setminus B_{\rho},\\[6pt]
	(\rho^{-1}I, \rho^{-1}I, \rho^{-1}) &\mbox{ in } B_{\rho}\setminus B_{\rho/2},\\[6pt]
	({F^{-1}_{\rho}}_*\eps_O, {F^{-1}_{\rho}}_*\mu_O, 0) &\mbox{ in } B_{\rho/2}.
	\end{array}\right.
	\end{equation}

Recall that for $\omega> 0$, a solution $(E, H) \in [H_{\loc}(\curl, \R^3\setminus B_R)]^2$, for some $R> 0$, of the Maxwell equations 
\[
\begin{cases}
\nabla \times E = i \omega H  &\text{ in } \mathbb{R}^3\setminus B_R,\\[6pt]
\nabla \times H = -i \omega E  &\text{ in } \mathbb{R}^3\setminus B_R
\end{cases}
\]
is called radiating if it satisfies one of the (Silver-M\"uller) radiation conditions
\begin{equation}\label{SM-condition}
H \times x - |x| E = O(1/|x|) \quad   \mbox{ or } \quad  E\times x + |x| H = O(1/|x|) \mbox{ as } |x| \to + \infty. 
\end{equation}
Here and in what follows, for $\alpha \in \mR$, $O(|x|^\alpha)$ denotes a quantity whose norm is bounded by $C |x|^\alpha$ for some constant $C>0$.

Throughout this section, we assume 
\begin{equation}
\dive \JJJ = 0 \quad \mbox{ and } \quad \supp \JJJ \subset B_{R_0} \setminus B_2, 
\end{equation}
for some $R_0 > 2$. One sees later (in Section~\ref{sc-proof-main}) that if $(\hEE_c, \hat \HH_c)$ and $(\hat \EE, \hat \HH)$ are the corresponding Fourier transform with respect to $t$ of $(\EE_c, \HH_c)$ and $(\EE, \HH)$ in \eqref{equ:wcloak}-\eqref{equ:ncloak}
and if one defines  $(\hEE_{\rho}, \hHH_{\rho}) = (DF_\rho^T \hEE_c, DF_\rho^T \hHH_c) \circ F_\rho$ in $\mR^3$ then $(\hEE, \hHH)$ and $(\hEE_\rho, \hHH_\rho)$ satisfy \eqref{Sys-F-1} and \eqref{Sys-F-1-rho} respectively (for some $\JJJ$). This is the motivation for the introduction of $(\EEE, \HHH) $ and $(\EEE_\rho, \HHH_\rho)$. 

The goal of this section is to derive estimates for $(\EEE_\rho, \HHH_\rho) - (\EEE, \HHH)$ in which the dependence on the frequency $\omega$ and $\rho$ is explicit. More precisely, we establish the following three results. 

\begin{proposition}\label{pro-low}
	Let $0<\rho<\rho_0$ and $0< \omega < \omega_0$.  We have
	\begin{equation}
	\label{thm-aux-low-claim}
	\|(\EEE_{\rho}, \HHH_{\rho}) - (\EEE, \HHH)\|_{L^2(B_{R}\setminus B_{2})} \leq C_R \omega^{-1} \rho^3\|\JJJ\|_{L^2(\R^3)},
	\end{equation}
	for some positive constant $C_R$ depending only on $R_0$, $R$, $\omega_0$, and $\rho_0$.
\end{proposition}

\begin{proposition}\label{prop-av}
Let $0<\rho<\rho_0$ and $0< \omega_0 \le \omega  \le   \omega_1 \rho^{-1}$, and  assume that  $\rho_0$ is small enough and $\omega_0$ is large enough.  We have, for $R> 2$, 
	\begin{equation}
	\label{thm-aux-av-claim}
	\|(\EEE_{\rho}, \HHH_{\rho}) - (\EEE, \HHH)\|_{L^2(B_{R}\setminus B_{2})} \leq C_R\omega^3  \rho^3\|\JJJ\|_{L^2(\R^3)},
	\end{equation}
	for some positive constant $C_R$ depending only on $R, R_0$, $\omega_0$, and $\omega_1$.
\end{proposition}

\begin{proposition}
	\label{pro-high}
	Let $0<\rho<1$,  $\omega_1>0$, and $\omega > \omega_1\rho^{-1}$. We have, for $R> 2$, 
	\begin{equation}
	\label{lem-high-1-claim}
	\|(\EEE_{\rho}, \HHH_{\rho}) - (\EEE, \HHH)\|_{L^2(B_R\setminus B_2)} \leq C_R \omega^{17/2}\rho^3\|\JJJ\|_{L^2(\R^3)}, 
	\end{equation}
	for some positive constant $C_R$ depending only on $R_0, R$,  and $\omega_1$.
\end{proposition}

To motivate the analysis in this section,  we define 
 \begin{equation}
	(\bE_{\rho}, \bH_{\rho}) = 
	\begin{cases}
	(\EEE_{\rho}, \HHH_{\rho}) - (\EEE, \HHH)  &\mbox{ in } \R^3\setminus B_\rho,\\[6pt]
	(\EEE_{\rho}, \HHH_{\rho}) & \mbox{ in } B_{\rho},
	\end{cases}
	\end{equation}
	and set  
\begin{equation}
	(\tbE_{\rho}, \tbH_{\rho}) = (\bE_{\rho}, \bH_{\rho}) (\rho \,  \cdot \, ) \mbox{ in } \mR^3. 
\end{equation}
Then, $(\tbE_{\rho}, \tbH_{\rho}) \in [L^2_{\loc}(\mR^3)]^6$ with $(\tbE_{\rho}, \tbH_{\rho}) \in \cap_{R > 1} H(\curl, B_R \setminus \partial B_1)$  is the unique radiating solution of 
\begin{equation}\label{eq-tEH-1}
	\begin{cases}
		\nabla\times \tbE_{\rho} = i\omega\tilde\mu_{\rho} \tbH_{\rho} & \mbox{ in } \R^3\setminus \partial B_1,\\[6pt]
		\nabla\times \tbH_{\rho} = -i\omega\tilde \eps_{\rho} \tbE_{\rho} +\tilde\sigma_{\rho} \tbE_{ \rho}  &\mbox{ in } \R^3\setminus \partial B_1, \\[6pt]
		[\tbE_{\rho}\times\nu]  = - \EEE(\rho \,  \cdot  \, )\times\nu  & \mbox{ on } \partial B_1,\\[6pt]
		  [\tbH_{\rho}\times\nu] = - \HHH(\rho \,  \cdot \, )\times\nu   &\mbox{ on } \partial B_{1},
	\end{cases}
	\end{equation}
where
\begin{equation}\label{eps-mu-rho-tilde}
(\tilde \eps_{\rho}, \tilde \mu_{\rho}, \tilde\sigma_{\rho}):= \left\{\begin{array}{cl}
(\rho I, \rho I, 0) &\mbox{ in } \R^3\setminus B_{1},\\[6pt]
(I,  I, 1) & \mbox{ in } B_{1}\setminus B_{1/2},\\[6pt]
(\eps_O, \mu_O, 0) & \mbox{ in } B_{1/2}.
\end{array}\right. 
\end{equation}
Here and in what follows for a  smooth, bounded, open subset $D$ of $\mR^3$, we denote $[u]: = u|_{\ext} - u|_{\inte}$ on $\partial D$ for an appropriate (vectorial) function $u$. 

\medskip 
We will  study \eqref{eq-tEH-1}  and using this to derive estimates for $(\EEE_{\rho}, \HHH_{\rho}) - (\EEE, \HHH)$ in the following three subsections. 

\subsection{Low and moderate frequency analysis - Proof of Proposition~\ref{pro-low}}\label{sc-low}

This section is devoted to the proof  of Proposition~\ref{pro-low} and contains two subsections. In the first subsection, we present several useful lemmas and the proof of Proposition \ref{pro-low} is given in the second subsection. 

\subsubsection{Some useful lemmas}


We first recall the following known result which is the basic ingredient for the variational approach.  

 \begin{lemma} \label{lem:compact}
Let $D$ be a smooth, bounded, open  subset of $\R^3$ and let $\epsilon$ be a measurable,  symmetric,   uniformly elliptic, matrix-valued function defined in $D$. Assume that one of the following two conditions holds:
\begin{enumerate}
\item[i)]  $(u_n)_{n\in \N}\subset H(\curl, D)$ is a bounded sequence in $H(\curl, D)$ such that 
\[\big(\dive(\epsilon u_n) \big)_{n\in \N}\text{ converges in }  H^{-1}(D) \text{ and } \big(u_n\times \nu \big)_{n\in \N} \text{ converges  in } H^{-1/2}(\partial D).\]

\item[ii)] $(u_n)_{n\in \N}\subset H(\curl, D)$ is a bounded sequence in $H(\curl, D)$  such that 
 \[\big(\dive(\epsilon u_n )\big)_{n\in \N} \text{ converges in }  L^2(D) \text{ and } \big((\epsilon u_n ) \cdot \nu \big)_{n\in \N}\text{ converges  in } H^{-1/2}(\partial D).\]
\end{enumerate}

There exists a subsequence of $(u_n)_{n\in \N}$ which converges in $[L^2(D)]^3$.
\end{lemma}

The conclusion of Lemma~\ref{lem:compact} under condition $i)$ is  \cite[Lemma 1]{HM1} and has its roots in \cite{Haddar, Costabel, Weber}.  The conclusion of Lemma~\ref{lem:compact} under condition $ii)$ can be obtained in  the same way.


In what follows, the following notations are used 
\begin{equation*}
H^{-1/2}(\dive_\Gamma, \Gamma): = \Big\{ \phi \in [H^{-1/2}(\Gamma)]^3; \; \phi \cdot \nu = 0 \mbox{ and } \dive_\Gamma \phi \in H^{-1/2}(\Gamma) \Big\},
\end{equation*}
\begin{equation*}
\| \phi\|_{H^{-1/2}(\dive_\Gamma, \Gamma)} : = \| \phi\|_{H^{-1/2}(\Gamma)} +  \| \dive_\Gamma \phi\|_{H^{-1/2}(\Gamma)}.
\end{equation*}

\medskip 
We have 
\begin{lemma} \label{lem-T-1-***} Let $0<\omega< \omega_0$ and $D$ be a simply connected, bounded, open subset of $\mR^3$ of class $C^1$, and  
denote  $\Gamma = \partial D$. Let $ h \in H^{-1/2}(\dive_\Gamma, \Gamma)$ and $E \in H(\curl, D)$. 
	We have 
			\begin{equation}\label{lem-T-1-c1}
			\Big| \int_{\Gamma} \bar E  \cdot h \, ds \Big| \le C\Big(\omega \|E\|_{L^2(D)} + \|\nabla\times E\|_{L^2(D)}\Big)\left(\|h\|_{H^{-1/2}(\Gamma)}+ \omega^{-1}\|\dive_{\Gamma}h\|_{H^{-1/2}(\Gamma)}\right),
			\end{equation}
	for some positive constant $C$ depending only on $D$ and $\omega_0$.
\end{lemma}

Here and in what follows, $\bar u$ denotes the complex conjugate of $u$. 

\begin{proof} Let $(E^0, H^0)\in [H(\curl, D)]^2$ be the unique solution to 
	\[
	\begin{cases}
	\nabla\times E^0 = i\omega(1+ i) H^0 &\mbox{ in } D,\\[6pt]
	\nabla\times H^0 = -i\omega(1+i) E^0 &\mbox{ in } D,\\[6pt]
	E^0\times \nu = h &\mbox{ on } \Gamma.
	\end{cases}
	\] 
	We prove by contradiction that  
	\begin{equation}\label{lem-T-1-e2}
	\|(E^0, H^0)\|_{L^2(D)} \leq C\left(\|h\|_{H^{-1/2}(\Gamma)} + \omega^{-1}\|\dive_{\Gamma} h\|_{H^{-1/2}(\Gamma)}\right)
	\end{equation}
	for some positive constant $C$ depending only on $\omega_0$. Assume that there exist sequences $((E_n, H_n))\subset [H(\curl, D)]^2$, $(\omega_n)\subset (0, \omega_0)$ and $(h_n) \subset H^{-1/2}(\dive_{\Gamma}, \Gamma)$ such that
	\begin{equation}\label{T-1-e1}
	\|(E_n, H_n)\| = 1 \mbox{ for all } n,
	\end{equation}
	\begin{equation}
	\|h_n\|_{H^{-1/2}(\Gamma)} + \omega_n^{-1}\|\dive_{\Gamma}h_n\|_{H^{-1/2}(\Gamma)} \mbox{ converges to } 0,
	\end{equation}
	and
	\begin{equation}\label{T-1-e2}
	\begin{cases}
	\nabla\times E_n = i\omega_n(1+i)H_n  &\mbox{ in } D, \\[6pt]
	\nabla\times H_n = -i\omega_n(1+i)E_n & \mbox{ in } D, \\[6pt]
	E_n\times\nu = h_n &\mbox{ in } \Gamma.
	\end{cases}
	\end{equation} 
	Without loss of generality, one can assume that $\omega_n \to \omega^*$. Applying  Lemma \ref{lem:compact}, one might  assume that $(E_n, H_n)$ converges to some $(E, H) \in [L^2(D)]^6$. We only consider the case $\omega_* = 0$, the case where $\omega_*> 0$ is  standard. Then 
	\[
	\begin{cases}
	\nabla\times E = 0 &\mbox{ in } D,\\[6pt]
	\dive E = 0 &\mbox{ in } D,\\[6pt]
	E\times \nu = 0 &\mbox{ on } \Gamma, 
	\end{cases} \quad \mbox{ and } \quad \begin{cases}
	\nabla\times H = 0 &\mbox{ in } D,\\[6pt]
	\dive H = 0 &\mbox{ in } D,\\[6pt]
	H\cdot \nu = 0 &\mbox{ on } \Gamma.
	\end{cases}
	\] 
	We also have, for each connected component $\Gamma_j$ of $\Gamma$,  
	\[
	\int_{\Gamma_j} E\cdot\nu \, ds= \lim\limits_{n\to \infty}\int_{\Gamma_j} E_n\cdot\nu\, ds = \lim\limits_{n\to \infty}\Big[\frac{1}{-i\omega_n(1+i)}\int_{\Gamma_j} (\nabla\times H_n)\cdot\nu\, ds\Big] = 0. 
	\]
Since $D$ is simply connected, it follows (see, e.g., \cite[Theorems 2.9 and 3.1]{Girault}) that  $E = \nabla \times \xi_E$ and $H  = \nabla \xi_H$ for some $\xi_E, \,  \xi_H \in H^1(D)$. We derive from the systems of $E$ and $H$ that 
$$
\int_{D} |\nabla \times \xi_E|^2 \, dx = 0 \quad \mbox{ and } \quad \int_{D} |\nabla \xi_H|^2 \, dx= 0. 
$$ 
This yields that $E = H = 0$ in $D$. We have a contradiction. Therefore, \eqref{lem-T-1-e2} is proved.

	We have
	\begin{align*}
	\int_{\Gamma}\bar E\cdot h \, ds &=\int_{\Gamma}\bar E\cdot (E^0 \times  \nu)  \, ds =  \int_{D}(\nabla\times \bar E)\cdot E^0\, dx - \int_{D}\bar E \cdot (\nabla\times E^0)\, dx  \; (\mbox{integration by parts}) \\
	\nonumber& = \int_{D}(\nabla\times \bar E)\cdot E^0\, dx - i\omega(1+i) \int_{D}\bar E \cdot H^0\, dx.
	\end{align*}
	It follows from H\"older's inequality and \eqref{lem-T-1-e2}  that
	\begin{align*}
	\left|\int_{\Gamma}\bar E\cdot h \, ds\right| & \leq \Big(\omega\|E\|_{L^2(D)} + \|\nabla \times E\|_{L^2(D)}\Big)\|(E^0, H^0)\|_{L^2(D)}\\
	& \leq C\Big(\omega\|E\|_{L^2(D)} + \|\nabla \times E\|_{L^2(D)}\Big)\left(\|h\|_{H^{-1/2}(\Gamma)} + \omega^{-1}\|\dive_{\Gamma} h\|_{H^{-1/2}(\Gamma)}\right), 
	\end{align*}
	which is  \eqref{lem-T-1-c1}.  
\end{proof}

The following simple result is used in our analysis. 

\begin{lemma} \label{lem-T-1}
Let $D$ be a $C^1$ bounded open subset of $\mR^3$ and denote $\Gamma = \partial D$. Let $ h \in H^{-1/2}(\dive_\Gamma, \Gamma)$ and $u \in H(\curl, D)$. 
	We have 
	\begin{equation} \label{lem-T-1-est}
	\Big| \int_{\Gamma} \bar u  \cdot h \Big| \le C \|u \|_{H(\curl, D)} \|h\|_{H^{-1/2}(\dive_\Gamma, \Gamma)}. 
	\end{equation}
	for some positive constant $C$  independent of $h$ and $u$.
\end{lemma}

\begin{proof} The result is standard. For the convenience of the reader, we present the proof.  By the trace theory, see, e.g., \cite{Alonso, BCS}, there exists $\phi \in H(\curl, D)$ such that 
\[
\phi\times\nu = h \mbox{ on } \Gamma \quad \mbox{ and } \quad \|\phi\|_{H(\curl, D)} \leq C\|h\|_{H^{-1/2}(\dive_{\Gamma}, \Gamma)}
\]
for some positive constant $C$ depending only on $D$. Then, by integration by parts, we have
\[
\int_{\Gamma} \bar u  \cdot h = \int_{\Gamma} \bar u  \cdot (\phi\times\nu) =  \int_D \nabla\times \bar u \cdot \phi - \int_D \bar u \cdot \nabla\times \phi. 
\]
The conclusion follows by H\"older's inequality.
\end{proof}

We next present an estimate for the exterior domain in the  small and moderate frequency regime. 

\begin{lemma}
\label{lem:ex-123}
Let $R_0>2$, $0 < k < k_0$, and  $D \subset B_1$ be a smooth  open subset of $\mR^3$ such that $\mR^3 \setminus D$ is connected. Let $(f_1, f_2) \in [L^2(\mR^3)]^6$ with support in $B_{R_0} \setminus D$, and  assume that  $(E, H) \in [\cap_{R> 1}H(\curl, B_R \setminus D)]^2$ is a radiating solution of 
\begin{equation*}
\begin{cases}
\nabla \times E  =  i k  H + f_1  & \text{ in } \mR^3\setminus \bar D,\\[6pt]
 \nabla \times H = -  ik  E + f_2   & \text{ in } \mR^3\setminus \bar D. 
\end{cases}
\end{equation*}
We have, for $R>2$, 
\begin{equation}\label{lem:ex0-conclusion-12*}
\|(E,  H)\|_{L^2(B_R\setminus D)} \leq C_R \Big(  \|(E \times \nu,  H\times \nu)\|_{H^{-1/2}(\partial D)} + \| (f_1,  f_2) \|_{L^2} + k^{-1} \| (\dive f_1, \dive f_2) \|_{L^2} \Big), 
\end{equation}
for some positive constant $C_R$ depending only on $D$, $k_0$, $R_0$,  and $R$.  
\end{lemma}

\begin{proof} By the Stratton-Chu formula, we have,  for $x \in \mR^3$ with $|x| > R_0 + 1$,  
\begin{multline*}
E(x)=   \int_{\partial B_{R_0 + 1/2}} \nabla_x G_k (x, y) \times \big( \nu(y) \times E (y) \big)  dy \\[6pt]
    + i k  \int_{\partial B_{R_0 + 1/2}}\nu(y)\times H(y)G_k(x,y)dy - \int_{\partial B_{R_0 + 1/2}}\nu(y)\cdot E(y) \nabla_x G_k (x, y)dy,
\end{multline*}
and 
\begin{multline*}
H(x)=   \int_{\partial B_{R_0 + 1/2}} \nabla_x G_k (x, y) \times \big( \nu(y) \times H (y) \big)  dy \\[6pt]
    - i k  \int_{\partial B_{R_0 + 1/2}}\nu(y)\times E(y)G_k (x,y)dy - \int_{\partial B_{R_0 + 1/2}}\nu(y)\cdot H(y) \nabla_x G_k (x, y)dy,
\end{multline*}
where 
\begin{equation}\label{def-G}
G_{k}(x, y) = \frac{e^{i k |x - y|}}{4 \pi |x -y|} \mbox{ for }  x \neq y. 
\end{equation}
It follows that, for $R > R_0 + 1$, 
\begin{equation}\label{SC-1***}
\|(E, H)\|_{L^2(B_R\setminus D)} \leq C_R \|(E,  H)\|_{L^2(B_{R_0 + 1} \setminus D)}.
\end{equation}
Hence, it  suffices to prove \eqref{lem:ex0-conclusion-12*} for $R =R_0 + 1$ by contradiction.  Assume that there exist  sequences $(k_n) \subset (0, k_0)$, $\big((f_{1, n}, f_{2, n}) \big) \subset [L^2(\mR^3 \setminus D)]^6$ with support in $B_{R_0} \setminus D$,  and  $\big((E_n, H_n)\big) \subset [\cap_{R> 1}H(\curl, B_R \setminus D)]^2$ such that $\|(E_n,  H_n)\|_{L^2(B_{R_0  + 1}\setminus D)}  = 1$, 
\begin{equation}\label{lem:ex0-a}
\lim_{n \to + \infty}  \Big( \|(E_n \times \nu , H_n \times \nu)\|_{H^{-1/2}(\partial D)} +  \| (f_{1, n}, f_{2, n}) \|_{L^2} + k_n^{-1} \| ( \dive f_{1, n}, \dive f_{2, n}) \|_{L^2}  \Big) = 0, 
\end{equation}
and 
\begin{equation*}
\begin{cases}
\nabla \times E_n  =  i k_n  H_n + f_{1, n} & \text{ in } \mR^3\setminus \bar D,\\[6pt]
 \nabla \times H_n = -  ik_n  E_n + f_{2, n}   & \text{ in } \mR^3\setminus \bar D. 
\end{cases}
\end{equation*}

Without loss of generality, one might assume that $k_n \to k_*$ as $n \to + \infty$. Using Lemma~\ref{lem:compact},  \eqref{SC-1***}, and \eqref{lem:ex0-a},  one can assume that $(E_n, H_n)$ converges to $(E, H)$ in $[L^2(B_R \setminus D)]^6$. We first consider the case $k_* = 0$. We have   
\begin{equation}\label{lem:ex-12-1}
\begin{cases}
\nabla \times E = 0 & \text{ in } \mR^3\setminus \bar D,\\[6pt]
E \times \nu = 0 &\mbox{ on } \partial D, 
\end{cases}
\quad \quad 
\begin{cases}
\nabla \times H  = 0 & \text{ in } \mR^3\setminus \bar D,\\[6pt]
H \times \nu = 0 &\mbox{ on } \partial D, 
\end{cases}
\end{equation}
\begin{equation}\label{lem:ex-12-2}
\dive E  = 0   \text{ in } \mR^3\setminus \bar D \quad \dive H = 0   \text{ in } \mR^3\setminus \bar D, 
\end{equation}
and 
\begin{equation}\label{lem:ex-12-4}
|E(x)| = O(|x|^{-2}) \mbox{ and } \quad |H(x)| = O(|x|^{-2})  \mbox{ for large $x$}. 
\end{equation}
Assertion~\eqref{lem:ex-12-4} can be derived again from the Stratton-Chu formula using the fact that $\lim_{n \to + \infty} k_n = 0$. It follows from \eqref{lem:ex-12-1}, \eqref{lem:ex-12-2},  and \eqref{lem:ex-12-4} that (see, e.g., \cite[Lemma 3.5]{HM9}, \cite[Chapter I]{Girault}) 
\[
E = H = 0 \mbox{ in } \R^3\setminus D.
\]
We have a contradiction with the fact $\| (E_n, H_n) \|_{L^2(B_{R_0 + 1} \setminus D)} = 1$.  

We next consider the case $k_*>0$. In this case, we have $(E, H)$ satisfies the radiating condition and 
\begin{equation*}
\begin{cases}
\nabla \times E  =  i k_*  H   & \text{ in } \mR^3\setminus \bar D,\\[6pt]
 \nabla \times H = -  ik_*  E    & \text{ in } \mR^3\setminus \bar D, \\[6pt]
 E \times \nu = H \times \nu = 0 & \mbox{ on } \partial D. 
\end{cases}
\end{equation*}
One also reaches  $(E, H) = (0, 0)$ in $\mR^3 \setminus D$ and obtains a contradiction. 
\end{proof}

In the same spirit, we have 

\begin{lemma}
\label{lem:ex-12}
Let  $0< \rho <  \rho_0$,  $0<  \omega  < \omega_0$, $1/2 < r < 1$, and $R_0> 2$.  Let 
$h = (h_1, h_2)\in [H^{-1/2}(\dive_{\partial B_1}, \partial B_1)]^2$. Assume that  $(E, H) \in  [L^2_{\loc}(\mR^3 \setminus B_r)]^6$ with  $(E, H) \in [\cap_{R> 1}H(\curl, (B_R \setminus B_r) \setminus \partial B_1)]^2$ is a radiating solution  of
\begin{equation*}
\begin{cases}
\nabla \times E  = i \omega \tilde \mu_\rho H  & \text{ in } (\mR^3\setminus \bar B_r) \setminus \partial B_1,\\[6pt]
 \nabla \times H = -  i\omega \tilde \eps_\rho E  +  \tilde \sigma_\rho E  & \text{ in } (\mR^3\setminus \bar B_r) \setminus \partial B_1,\\[6pt]
[E\times \nu] = h_1, [H\times\nu] = h_2 &\mbox{ on } \partial B_1.
\end{cases}
\end{equation*}
We have, for $R>2$, 
\begin{multline*}
\|(E, H)\|_{L^2(B_R\setminus B_r)} \leq C_R\Big(\|(E \times \nu ,  H\times \nu)\|_{H^{-1/2}(\partial B_r)} + \|(h_1, h_2)\|_{H^{-1/2}( \partial B_1)} \\[6pt]
 + \omega^{-1} \|(\dive_{\partial B_1} h_1, \dive_{\partial B_1} h_2)\|_{H^{-1/2}(\partial B_1)}  \Big), 
\end{multline*}
for some positive constant $C_R$ independent of $(h_1, h_2)$, $(f_1, f_2)$, $\rho$, and $\omega$.  
\end{lemma}

\begin{proof}  As argued in the proof of Lemma~\ref{lem:ex-123}, by Stratton-Chu's formulas,  it suffices to prove
\begin{multline*}
\|(E,  H)\|_{L^2(B_2\setminus B_r)} \le C_R \Big( \|(E \times \nu ,  H\times \nu)\|_{H^{-1/2}(\partial B_r)} + \|(h_1, h_2)\|_{H^{-1/2}( \partial B_r)}  \\[6pt]
 +   \omega^{-1} \|(\dive_{\partial B_1} h_1, \dive_{\partial B_1} h_2)\|_{H^{-1/2}(\partial B_1)}  \Big), 
\end{multline*} 
 by contradiction. Assume that there exist  sequences $(\omega_n) \subset (0, \omega_0)$, $\big((h_{1, n}, h_{2, n}) \big)\subset [H^{-1/2}(\dive_{\Gamma}, \partial B_1)]^2$, $\big((f_{1, n}, f_{2, n}) \big) \subset L^2(\mR^3 \setminus B_r)$ with support in $B_{1} \setminus B_r$,  and  $\big((E_n, H_n)\big) \subset [\cap_{R> 1}H(\curl, B_R \setminus D)]^2$ such that 
\begin{equation}\label{tttt-1}
\|(E_n,  H_n)\|_{L^2(B_{2}\setminus B_r)}  = 1,
\end{equation} 
\begin{multline}\label{tttt-111}
\lim_{n \to + \infty} \Big( \|(E_n \times \nu ,  H_n \times \nu)\|_{H^{-1/2}(\partial B_r)} + \|(h_{1, n} , h_{2, n})\|_{H^{-1/2}( \partial B_r)}  \\[6pt]
 + \omega_n^{-1} \|(\dive_{\partial B_1} h_{1, n}, \dive_{\partial B_1} h_{2, n})\|_{H^{-1/2}(\partial B_1)} \Big)= 0, 
\end{multline}
and 
\begin{equation*}
\begin{cases}
\nabla \times E_n  =  i \omega_n  \tilde \mu_{\rho_n}H_n  & \text{ in } (\mR^3\setminus \bar B_r)\setminus \partial B_1,\\[6pt]
 \nabla \times H_n = -  i\omega_n  \tilde \eps_{\rho_n} E_n  + \tilde \sigma_{\rho_n} E_n   & \text{ in } (\mR^3\setminus \bar B_r)\setminus \partial B_1
 ,\\[6pt]
[E_n \times \nu] = h_{1, n}, [H_n  \times\nu] = h_{2, n} &\mbox{ on } \partial B_1.
\end{cases}
\end{equation*}

Without loss of generality, one might assume that $\omega_n \to \omega_*$ and $\rho_n \to \rho_*$ as $n \to + \infty$. We first consider the case $\rho_* = 0$. Since, as $n \to + \infty$,  
$$
(- i \omega_n + 1)E_n \cdot \nu |_{int} = - i \omega_n \rho_n E_n \cdot \nu|_{ext}  -  \dive_{\partial B_1}h_{2, n}  \to 0 \mbox{ in } H^{-1/2}(\partial B_1)\quad 
$$
and 
$$ 
H_n \cdot \nu |_{int} = \rho_n H_n \cdot \nu|_{ext}  - (i \omega_n)^{-1} \dive_{\partial B_1} h_{1, n}  \to 0 \mbox{ in } H^{-1/2}(\partial B_1), 
$$
using \eqref{tttt-111} and applying Lemma~\ref{lem:compact},  one can assume that $(E_n,  H_n)$ converges to $(E, H)$ in $L^2(B_1 \setminus B_r)$. Moreover,  
\begin{equation*}
\begin{cases}
\dive E =  \dive  H = 0 & \text{ in } B_1 \setminus \bar B_r,\\[6pt]
E \times \nu =  H \times \nu = 0 &\mbox{ on } \partial B_r, \\[6pt]
E \cdot \nu =  H \cdot \nu = 0 &\mbox{ on } \partial B_1. 
\end{cases}
\end{equation*}
It follows that $(E, H) = (0, 0)$ in $B_1 \setminus B_r$.  We derive that 
\begin{equation}\label{tttt-2}
\lim_{n \to + \infty} \|(E_n, H_n) \|_{L^2(B_1 \setminus B_r)} = 0
\end{equation}
and, by  \cite[Lemma A1]{Haddar},
\begin{equation*}
\lim_{n \to + \infty}  \|(E_n \times \nu, H_n \times \nu) |_{\inte} \|_{H^{-1/2}(\partial B_1 )} = 0.   
\end{equation*}
This yields 
\begin{equation}\label{tttt-3}
\lim_{n \to + \infty} \|(E_n \times \nu, H_n \times \nu)|_{\ext} \|_{H^{-1/2}(\partial  B_1 )} = 0. 
\end{equation}
This in turn implies,  by Lemma~\ref{lem:ex-123},  that
\begin{equation}\label{tttt-3}
\lim_{n \to + \infty} \|(E_n , H_n) \|_{L^2(B_2 \setminus  B_1 )} = 0. 
\end{equation}
Combining \eqref{tttt-1}, \eqref{tttt-2}, and \eqref{tttt-3}, we obtain  a contradiction.

We next consider the case $\rho_* > 0$. The proof in this case is similar to the one in Lemma~\ref{lem:ex-123} and omitted (see also \cite[Lemma 4]{HM1} for the case $\omega_*>0$). 
\end{proof}

\begin{remark} \rm The proof gives the following slightly sharper estimate (for small  $\omega$):
\begin{multline}\label{lem:ex0-conclusion-t12}
\|(E, H)\|_{L^2(B_R\setminus B_r)} \leq C_R\Big(\|(E \times \nu ,  H\times \nu)\|_{H^{-1/2}(\partial B_r)} + \|(h_1, h_2)\|_{H^{-1/2}( \partial B_r)} \\[6pt]
 +  \|(\omega^{-1}\dive_{\partial B_1} h_1, \dive_{\partial B_1} h_2)\|_{H^{-1/2}(\partial B_1)}  \Big). 
\end{multline}

\end{remark}

We are ready to give the main result of this section:

\begin{lemma}\label{lem-lossy-low}
	Let $0 < \rho < \rho_0$ and $0 < \omega < \omega_0$, and let $h_1, h_2\in H^{-1/2}(\dive_{\partial B_1}, \partial B_{1})$. Let $(E_{\rho}, H_{\rho}) \in [\cap_{R> 1}H(\curl, B_R\setminus \partial B_{1})]^2$ be the unique radiating solution of 
	\begin{equation}\label{lem-lossy-low-e0}
	\begin{cases}
	\nabla\times E = i\omega \tilde \mu_{\rho} H &\mbox{ in } \R^3 \setminus \partial B_1,\\[6pt]
	\nabla\times  H = -i\omega \tilde \eps_{\rho} E + \tilde \sigma_{\rho} E & \mbox{ in } \R^3 \setminus B_1,\\[6pt]
	[E \times \nu] = h_1, \,  [H \times \nu] = h_2 & \mbox{ on } \partial B_{1}. 
	\end{cases}
	\end{equation} 
	We have 
	\begin{equation*}
	\|(E,  H) \|_{L^{2}(B_2\setminus B_{2/3})}   \leq C \Big( \|(h_1,  h_2)\|_{H^{-1/2}(\partial B_{1})} + \omega^{-1} \|(\dive_{\partial B_1}h_1,  \dive_{\partial B_1}h_2)\|_{H^{-1/2}(\partial B_{1})} \Big), 
		\end{equation*}
	for some positive constant $C$ depending only on $\rho_0$ and $\omega_0$. 
\end{lemma}

\begin{proof} Multiplying  the first equation of \eqref{lem-lossy-low-e0} by $\tilde \mu_\rho^{-1}\nabla \times \bar E$ and integrating over $B_R \setminus \partial B_1$, we have, for $R> 1$, 
	\begin{multline*}
	\nonumber\int_{B_R \setminus \partial B_1} \tilde \mu_\rho^{-1}\nabla \times E  \cdot \nabla \times \bar E \, dx = i\omega\int_{B_R \setminus \partial B_1}H \cdot \nabla\times \bar E\, dx \\[6pt]
	 = i\omega\int_{B_R \setminus \partial B_1} ( -i\omega \tilde \eps_\rho E +\tilde \sigma_\rho E )  \cdot \bar E\, dx + i\omega \int_{\partial B_R} (H \times \nu)  \cdot \bar E \, dx   \\[6pt]
	 - i\omega   \int_{\partial B_1} (H \times \nu)|_{\ext} \cdot \bar E |_{\ext} -  (H \times \nu)|_{\inte} \cdot \bar E |_{\inte}.
	\end{multline*}
	Using the definition of  $\tilde \sigma_\rho$ and considering  the imaginary part, we have 
	\begin{equation}\label{lem-lossy-low-e3.2}
	\int_{B_1\setminus B_{1/2}}|E|^2\,dx  = \Re \left( \int_{\partial B_1}h_{2}\cdot \bar E |_{\ext}  \, dx   - \bar h_{1}\cdot H |_{\inte} \, d x \right) - \Re\int_{\partial B_R} (H\times \nu) \cdot \bar E\, dx.
	\end{equation}
		Letting $R\to \infty$ and  using the radiation condition, we derive from \eqref{lem-lossy-low-e3.2} that 
	\begin{align}\label{lem-lossy-low-e5.1}
	\int_{B_1\setminus B_{1/2}}|E|^2\,dx  &\le   \left|\int_{\partial B_1} h_{2} \cdot \bar E |_{\ext} - \bar h_{1} H |_{\inte} \, ds\right|\\
	\nonumber &  \le   \left|\int_{\partial B_1} h_{2} \cdot \bar E |_{\ext}\right| + \left|\int_{\partial B_1}\bar h_{1}\cdot H |_{\ext} \, ds\right| + \left|\int_{\partial B_1}(\bar h_{1}\times \nu)\cdot h_2 \, ds\right|.
	\end{align}
Applying Lemma \ref{lem-T-1-***} with $D  = B_{2} \setminus B_{1}$, we have
\begin{equation}\label{lem-lossy-low-e5.2}
 \left|\int_{\partial B_1}h_{2}\cdot \bar E |_{\ext} \, ds\right|\leq  C \omega \|(E,  H)\|_{L^2(B_2\setminus B_{1})} \Big(\|h_2\|_{H^{-1/2}(\partial B_1)}+ \omega^{-1}\|\dive_{\Gamma} h_2\|_{H^{-1/2}(\partial B_1)}\Big)
\end{equation}
and
\begin{equation}\label{lem-lossy-low-e5.2-*}
 \left|\int_{\partial B_1}h_{1}\cdot \bar H |_{\ext}  \, ds\right|\leq  C  \omega\|( E,  H)\|_{L^2(B_2\setminus B_{1})} \Big(\|h_1\|_{H^{-1/2}(\partial B_1)}+ \omega^{-1}\|\dive_{\Gamma} h_1\|_{H^{-1/2}(\partial B_1)}\Big). 
\end{equation}
Applying Lemma \ref{lem-T-1}, we obtain
\begin{equation}\label{lem-lossy-low-e5.2-**}
\left|\int_{\partial B_1}(\bar h_{1}\times\nu)\cdot h_2 \, ds\right|\leq  C \|(h_1, h_2)\|^2_{H^{-1/2}(\dive_{\partial B_1}, \partial B_1)}. 
\end{equation}
Denote
\[
M = \|(h_1, h_2)\|_{H^{-1/2}(\partial B_1)}+ \omega^{-1}\|(\dive_{\Gamma} h_1, \dive_{\Gamma} h_2)\|_{H^{-1/2}(\partial B_1)}.
\]
Combining  \eqref{lem-lossy-low-e5.1},  \eqref{lem-lossy-low-e5.2}, \eqref{lem-lossy-low-e5.2-*} and \eqref{lem-lossy-low-e5.2-**} yields 
\begin{equation}\label{lem-lossy-low-e5.3}
	\int_{B_1\setminus B_{1/2}}|E|^2\,dx  \le \\ C \Big(\omega M \| (E, H) \|_{L^2(B_2 \setminus B_{1})} + M^2 \Big).
	\end{equation}
From the equations of $(E, H)$ in $B_{1} \setminus B_{1/2}$, we have 
\begin{equation*}
\Delta E + \omega^2E - i \omega E = 0 \mbox{ in } B_1\setminus B_{1/2}.
\end{equation*}
It follows from \eqref{lem-lossy-low-e5.3} that
\begin{equation}\label{lem-lossy-low-e5.4}
\|  E\|_{L^2(\partial B_{2/3})}^2 + \| \nabla E\|_{L^2(\partial B_{2/3})}^2 \le  \\ C \Big(\omega M\| (E,  H) \|_{L^2(B_2 \setminus B_{1})} + M^2 \Big), 
\end{equation}
which yields 
\begin{equation}\label{lem-lossy-low-e5.5}
\| (E, H)\|_{L^2(\partial B_{2/3})}^2 \le  C \Big( \omega^{-1} M\| (E,  H) \|_{L^2(B_2 \setminus B_{1})}  + \omega^{-2 }M^2\Big).
\end{equation}
Using \eqref{lem-lossy-low-e5.5} and applying Lemma~\ref{lem:ex-12} with $r = 2/3$, we derive that 
\begin{equation*}
\|(E, H)\|_{L^2(B_R \setminus  B_{2/3})}^2 \le C  \Big( \omega^{-1}M\| (E,  H) \|_{L^2(B_2 \setminus B_{1})}  + \omega^{-2}M^2\Big), 
\end{equation*}
and the conclusion follows. 
\end{proof}

We end this subsection with  

\begin{lemma}
	\label{lem-FF-1}
	Let $0< \rho < 1$ and $\rho\omega < k_0$, and 
	let $D \subset B_1$ be a smooth, open subset of $\mR^3$. 
	Assume that $(E, H)\in [\cap_{R > 2}H(\curl, B_R\setminus D)]^2$ is a radiating solution to the system
	\[
	\begin{cases}
	\nabla\times E = i\omega\rho H & \mbox{ in } \R^3\setminus D,\\[6pt]
	\nabla\times H = -i\omega\rho E & \mbox{ in } \R^3\setminus D.
	\end{cases}
	\]
	We have,  for $R \ge 1$ and $x\in B_{3R/\rho}\setminus B_{2R/\rho}$, 
	\begin{equation}\label{lem-FF-1-c}
	|E(x), H(x)|\leq C_R \rho^3   (\omega^2 + 1) \|(E, H)\|_{L^2(B_2\setminus D)}, 
	\end{equation}
	for some positive constant $C$ depending only on  $k_0$ and $R$.
\end{lemma}

\begin{proof} We only prove \eqref{lem-FF-1-c} for $E$, the proof $H$ is similar. By Stratton-Chu's formula, we have,  for $x \in \mR^3 \setminus \bar B_1$,  
\begin{multline}\label{SC-E-FF}
E(x)=   \int_{\partial B_1} \nabla_x G_k (x, y) \times \big( \nu(y) \times E (y) \big)  dy \\[6pt]
    + i \omega \rho \int_{\partial B_1}\nu(y)\times H(y)G_k(x,y)dy - \int_{\partial B_1}\nu(y)\cdot E(y) \nabla_x G_k(x, y)dy,
\end{multline}
where $k = \omega\rho$ and $G_{k}$ is given in \eqref{def-G}.

Let $(\widetilde E, \widetilde H)\in [H(\curl, B_1)]^2$ be the unique solution to the system
\begin{equation}
\label{FF-3}
\begin{cases}
\nabla\times \widetilde E = i\omega\rho (1 + i) \widetilde H & \mbox{ in } B_1,\\[6pt]
\nabla\times \widetilde H = -i\omega\rho (1 + i) \widetilde E & \mbox{ in } B_1, \\[6pt]
\widetilde E\times \nu = E\times \nu & \mbox{ on } \partial B_1. 
\end{cases}
\end{equation}
By a contradictory  argument, see, e.g.,  \cite{HM9} (see also the proof of Lemma~\ref{lem-lossy-low}),  we obtain
\begin{equation}\label{FF-4}
 \|(\widetilde E,  \widetilde H)\|_{L^2(B_1)} \leq C \|E\times \nu_{\ext}, H \cdot \nu|_{\ext}\|_{H^{-1/2}(\partial B_1)}.
\end{equation}
Since
\[\left|\int_{\partial B_1}E\times \nu \,ds\right|= \left|\int_{\partial B_1}\widetilde E\times \nu \,ds\right| = \left|\int_{B_1}\nabla\times \widetilde E \, dx\right| = \left|\int_{B_1} \omega\rho (1 + i) \widetilde H dx\right|, \]
we obtain
\begin{equation}\label{FF-1-p1}
\left| \int_{\partial B_1} E \times \nu \, ds\right|  \le C \omega \rho \| (E, H) \|_{L^2(B_2\setminus D)}. 
\end{equation}
Similarly, we have 
\begin{equation}\label{FF-1-p2}
\left| \int_{\partial B_1} H \times \nu \, ds\right|  \le C \omega \rho \| (E, H) \|_{L^2(B_2\setminus D)}. 
\end{equation}

One has 
\begin{equation}\label{FF-2}
\int_{\partial B_1} \nu \cdot E\, ds = \frac{1}{i \omega \rho} \int_{\partial B_1} \nu \cdot \nabla \times H \,ds= 0. 
\end{equation}
Rewrite  \eqref{SC-E-FF} under the form  
\begin{align*}
& E(x)= \\[6pt] 
&   \int_{\partial B_1} \nabla_x G_k (x, 0) \times \big( \nu(y) \times E (y) \big)  dy  +  \int_{\partial B_1} \big(\nabla_x G_k (x, y) -  \nabla_x G_k (x, 0) \big) \times \big( \nu(y) \times E (y) \big)  dy  \\[6pt]
    & + i k \int_{\partial B_1}\nu(y)\times H(y) G_k(x,0)dy  + i k \int_{\partial B_1}\nu(y)\times H(y) \big( G_k(x,y) - G_k(x,0) \big)dy \\[6pt]
  & - \int_{\partial B_1}\nu(y)\cdot E(y) \nabla_x  G_k(x, 0) dy -  \int_{\partial B_1}\nu(y)\cdot E(y) \big( \nabla_x  G_k(x, y) - \nabla_x G(x, 0)  \big) dy. 
\end{align*}
Using the facts, for $|x| \in (2R/ \rho, 3R / \rho)$ and $y \in \partial B_1$,  
$$
|G_{k}(x, y)  -  G_k(x, 0)|  \le C (1 + \omega) \rho^2, \quad |\nabla G_{k}(x, y)  - \nabla G_k(x, 0)| \le C (1 + \omega^2) \rho^3,
$$
$$
\|E\|_{L^2(\partial B_1)} \leq C \|E\|_{L^2(B_2\setminus D)},  \quad \mbox{ and } \quad 
\|H\|_{L^2(\partial B_1)} \leq C \|H\|_{L^2(B_2\setminus D)}, 
$$
we derive the conclusion from \eqref{FF-1-p1}, \eqref{FF-1-p2}, and  \eqref{FF-2}. 
\end{proof}

\subsubsection{Proof of Proposition \ref{pro-low}} Applying Lemma \ref{lem-lossy-low} to $(\tbE_{\rho}, \tbH_{\rho})$, defined in \eqref{eq-tEH-1}, we have 
	\begin{equation}\label{lem-low-e2}
	\|(\tbE_{\rho},  \tbH_{\rho})\|_{L^{2}(B_{2} \setminus B_1)} 
	\leq  C \omega^{-1}  \|\big(\EEE(\rho\, .) ,  \HHH(\rho\,.) \big) \|_{L^2(\partial B_1)}.
	\end{equation}
Since $\dive \JJJ = 0$, we have 
$$
\Delta \EEE + \omega^2 \EEE = - i \omega \JJJ \mbox{ in } \mR^3. 
$$
It follows that,  for $x\in B_2$, 
	\begin{equation}\label{STST}
	\EEE(x) = -  i \omega\int_{\R^3}\JJJ(y) G_{\omega}(x, y)\, dy  \quad \mbox{ and } \quad 
		\HHH(x) = -  \nabla_x\times \int_{\R^3} \JJJ(y) G_{\omega}(x, y)\, dy.
	\end{equation}
This yields
\begin{equation}\label{lem-low-e4}
	\| \big(\EEE(\rho .),  \HHH(\rho .) \big)\|_{L^{\infty}(\partial B_1)}\leq C  \|\JJJ\|_{L^2(\R^3)}.
	\end{equation}
	From \eqref{lem-low-e2} and \eqref{lem-low-e4}, we obtain 
		\begin{equation}
			\|(\tbE_{\rho},  \tbH_{\rho})\|_{L^2(B_{2} \setminus B_1)} \leq C  \omega^{-1} \|\JJJ\|_{L^2(\R^3)}.
		\end{equation}
	Applying Lemma \ref{lem-FF-1} to $(\tbE_{\rho}, \tbH_{\rho})$, we have, for $x  \in B_{3r/ \rho} \setminus B_{2 r/ \rho}$,  
	 \begin{equation*}
	\Big|\big(\tbE_{\rho}(x), \tbH_{\rho}(x) \big) \Big|  \le C_r \omega^{-1} \rho^3  \| \JJJ \|_{L^2(\R^3)} \mbox{ for } r> 1/ 2, 
	 \end{equation*}
Since $(\EEE_{\rho}, \HHH_{\rho}) - (\EEE, \HHH) = (\tbE_{\rho}, \tbH_{\rho}) (\rho^{-1} \, \cdot \, )$ in $\R^3\setminus B_2$, the conclusion follows. 
\qed


\subsection{Moderate and high frequency analysis - Proof of Proposition~\ref{prop-av}}\label{sc-av}
This section contains two subsections. In the first, we present several lemmas used in the proof of Proposition~\ref{prop-av} and in the second, the proof of Proposition \ref{prop-av} is given. 

\subsubsection{Some useful lemmas}  The main result of this subsection is Lemma \ref{lem-lossy} which is analogous  to Lemma \ref{lem:ex-12} though for  the moderate and high frequency regime. We begin with 

\begin{lemma}\label{lem-mor}
	Let $\omega > \omega_0$, and let $\Omega$ be a {\bf convex}, bounded subset of $\R^3$ of class $C^1$. Let $j \in H(\dive, \Omega)$, and let $u\in H(\curl, \Omega) \cap H(\dive, \Omega)$ be such that 
	\begin{equation}\label{lem-moz-e1}
		\nabla\times (\nabla\times u) - \omega^2 u =  j \mbox{ in } \Omega,
	\end{equation}
	and $u\cdot \nu, \,  (\nabla\times u)\cdot \nu \in L^2(\partial \Omega)$. Then
	\begin{multline}\label{lem-mor-claim-*}
	\|(\omega u\times\nu, (\nabla\times u)\times \nu)\|_{L^2(\partial \Omega)} \\[6pt]
	\leq C\Big( \|(\omega u, \nabla\times u)\|_{L^2(\Omega)}
	 + \|(\omega u\cdot\nu, (\nabla \times u) \cdot \nu)\|_{L^2(\partial \Omega)}  + \|j\|_{L^2(\Omega)} + \omega^{-1}\|\dive j\|_{L^2(\Omega)}\Big) ,
	\end{multline}
	for some positive constant $C$ depending only on $\Omega$ and $\omega_0$.
\end{lemma}
\begin{proof} The analysis is based on the multiplier technique. We first consider $\dive j = 0$.   Multiplying \eqref{lem-moz-e1} by $(\nabla\times \bar u)\times x$ and integrating over $\Omega$, we obtain
	\begin{multline}
		\label{lem-moz-e3}
		\int_{\Omega}j \cdot (\nabla\times \bar u)\times x\, dx = \int_{\Omega} \nabla\times(\nabla\times u)\cdot (\nabla\times \bar u)\times x\, dx - \omega^2\int_{\Omega}u \cdot (\nabla\times \bar u)\times x\, dx.
	\end{multline}
	Set 
	\[
	I_1 := - \omega^2 \int_\Omega u \cdot (\nabla\times \bar u)\times x\, dx
\quad 	\mbox{ and } \quad 
	I_2 := \int_{\Omega} \nabla\times(\nabla\times u)\cdot (\nabla\times \bar u)\times x\, dx.
	\]
We have 
	\begin{align*}
		I_1 &=  -\omega^2\int_{\Omega} u\cdot (\nabla\times \bar u)\times x\, dx =   \omega^2 \int_{\Omega} (\nabla\times \bar u)\cdot (u\times x)\, dx\\[6pt]
		&=   \omega^2 \int_{\Omega} \bar u\cdot \nabla\times(u \times x)\, dx  - \omega^2 \int_{\partial \Omega}(\bar u \times \nu)\cdot (u\times x)\, ds \quad (\mbox{by integration by parts}).
	\end{align*}
	Recall that,  for all $v\in [H^1(\Omega)]^3$, 
	\begin{equation}\label{lem-moz-e3.1}
	\nabla\times(v\times x) = - x\times(\nabla\times v) + v + \nabla (v\cdot x) - x\dive v  \quad \mbox{ in } \Omega.
	\end{equation}
	Using \eqref{lem-moz-e3.1}  and the fact $\dive  u = \dive j = 0$ in $\Omega$, we derive that 
	\begin{align*}
		I_1 
		&= - \omega^2\int_{\Omega}\bar u\cdot \big[x\times (\nabla\times u)\big]\, dx + \omega^2\int_{\Omega}|u|^2\,dx\\[6pt] 
		&\quad \quad \quad\quad  + \omega^2\int_{\Omega}\bar u\cdot\nabla (u \cdot x)\, dx - \omega^2 \int_{\partial \Omega}(\bar u \times \nu)\cdot (u\times x)\, ds\\[6pt]
		&= - \overline{I_1} + \omega^2\left[\int_{\Omega}|u|^2\,dx + \int_{\partial \Omega}(\bar u\cdot \nu)(u \cdot x)\, ds - \int_{\partial \Omega}(\bar u \times \nu)\cdot (u\times x)\, ds \right].
		\end{align*}
This implies
	\begin{equation}\label{lem-mor-e3}
	\Re{I_1} = \frac{\omega^2}{2}\left(\int_{\Omega}|u|^2\, dx +  \int_{\partial \Omega}(\bar u\cdot \nu)(u \cdot x)\, ds - \int_{\partial \Omega}(\bar u \times \nu)\cdot (u\times x)\, ds\right).
	\end{equation}
	Similarly, we have
	\begin{equation}\label{lem-mor-e4}
		\Re{I_2} = \frac{1}{2}\left(\int_{\Omega}|\nabla\times u|^2\, dx + \int_{\partial \Omega}(\nabla\times\bar u\cdot \nu)(\nabla\times u \cdot x)\, ds - \int_{\partial \Omega}\big((\nabla\times \bar u ) \times \nu \big)\cdot \big((\nabla\times u)\times x \big)\, ds\right).
	\end{equation}	
Combining  \eqref{lem-moz-e3}, \eqref{lem-mor-e3},  and \eqref{lem-mor-e4} yields 
	\begin{multline}\label{lem-mor-p1}
		\int_{\Omega}\omega^2|u|^2 + |\nabla\times u|^2\, dx  - \int_{\partial \Omega}\omega^2(\bar u \times \nu)\cdot (u\times x) + ((\nabla\times \bar u) \times \nu)\cdot ((\nabla\times u)\times x)\, ds  \\[6pt] +  \int_{\partial \Omega} \omega^2 (\bar u\cdot \nu)(u \cdot x) + (\nabla\times\bar  u\cdot \nu)(\nabla\times u \cdot x) \, ds  = 2\Re \left\{\int_{\Omega}j \cdot(\nabla\times\bar u)\times x\, dx\right\} .
	\end{multline}
This implies \eqref{lem-mor-claim-*} in the case where $\dive j = 0$ in $\Omega$. 

We next consider an arbitrary $\dive j$. Let $\phi \in H^1_0(\Omega)$ be the unique solution of
\[
\Delta \phi = \dive j \quad \mbox{ in } \Omega. 
\]
It is clear that
\begin{equation}\label{lem-mor-p12}
 \| \phi \|_{H^1(\Omega)}  \le C\|j\|_{L^2(\Omega)}
\end{equation}
and
\begin{equation}\label{lem-mor-p13}
\|\nabla\phi\times \nu\|_{L^2(\partial \Omega)} \leq C\|\phi\|_{H^2(\Omega)}\leq C \|\dive j\|_{L^2(\Omega)},
\end{equation}
for some positive constant $C$ depending only on $\Omega$.
 Set 
\begin{equation}\label{lem-mor-p11}
\tilde u = u - \omega^{-2}\nabla \phi\mbox{ in } \Omega.
\end{equation}
We have 
\[
\nabla\times \nabla\times \tilde u - \omega^2 \tilde u = j - \nabla \phi \mbox{ in } \Omega.
\]
Since $\dive(j - \nabla \phi) = 0$ in $\Omega$, applying the previous case to $\tilde u$, we obtain the conclusions from \eqref{lem-mor-p12}, \eqref{lem-mor-p13}, and \eqref{lem-mor-p11}.
\end{proof}

As a consequence of Lemma~\ref{lem-mor}, one has
 
\begin{corollary}\label{cor-mor} Let $\omega > \omega_0$. Let $j \in H(\dive, B_1\setminus B_{3/4})$, and let $(E, H)\in [H(\curl, B_1\setminus B_{3/4})]^2$ be such that $E\cdot\nu, H\cdot \nu \in [L^2(\partial B_1)]^3$. Assume that
\[
\begin{cases}
\nabla\times E = i\omega H &\mbox{ in } B_1\setminus B_{3/4},\\[6pt]
\nabla\times H = -i\omega E + j &\mbox{ in } B_1\setminus B_{3/4},
\end{cases}
\quad \mbox{ and } \quad 
\dive j = 0 \mbox{ in } B_1\setminus B_{3/4}.
\]
We have 
\begin{equation*}
\|(E \times \nu, H\times\nu)\|_{L^2(\partial B_1)} \leq C\Big(\|(E, H)\|_{L^2(B_1\setminus B_{3/4})} 
+ \|(E\cdot\nu, H\cdot\nu)\|_{L^2(\partial B_1)} + \|j\|_{L^2(B_1\setminus B_{3/4})}\Big), 
\end{equation*}
for some positive constant $C$ depending only on $\omega_0$.
\end{corollary}
\begin{proof}
Let $0\leq \phi \leq 1$ be a smooth function in $B_1$ such that $\phi(x) = 0 \mbox{ in } B_{4/5} \mbox{ and } \phi(x) = 1 \mbox{ in } B_1\setminus B_{5/6}$. Extend $u$ and $j$ by $0$ in $B_{3/4}$,  and set $u = \phi E \mbox{ in } B_1$. Then
\begin{equation}\label{cor-mor-e1}
\nabla\times \nabla \times u - \omega^2 u = i\omega \phi j + \nabla\times(\nabla \phi \times E) + \nabla \phi \times(\nabla\times E)  \mbox{ in } B_1.
\end{equation}
Since $\Delta E + \omega^2 E  = i\omega j$ in $B_1\setminus B_{3/4}$, we have
\begin{equation}\label{cor-mor-e2}
\|\nabla E\|_{L^2(B_{5/6}\setminus B_{4/5})} \leq C\omega \Big(\|E\|_{L^2(B_1\setminus B_{3/4})} + \|j\|_{L^2(B_{1}\setminus B_{3/4})} \Big). 
\end{equation}
Applying Lemma \ref{lem-mor}  and using \eqref{cor-mor-e1} and  \eqref{cor-mor-e2},  one obtains the conclusion.
\end{proof}

The main result of this section is the following lemma,  which is a variant of Lemma \ref{lem-lossy-low} in the case where $\omega_0 < \omega <\omega_1\rho^{-1}$. 

\begin{lemma}\label{lem-lossy}
	Let  $ 0 < \rho  < \rho_0$ and $0< \omega_0 <  \omega < \omega_1/ \rho$.  Suppose that $h_1, h_2\in L^{2}(\dive_{\Gamma}, \partial B_{1})$, and let $(E, H) \in [\cap_{R> 1}H(\curl, B_R\setminus \partial B_{1})]^2$ be the unique radiating solution to the system 
	\begin{equation*}		\begin{cases}
			\nabla\times E = i\omega \tilde \mu_\rho H &\mbox{ in } \R^3,\\[6pt]
			\nabla\times  H = -i\omega \tilde \eps_\rho E + \tilde\sigma_{\rho}E  & \mbox{ in } \R^3,\\[6pt]
			[E\times \nu] = h_1, \, [H\times \nu] = h_2 & \mbox{ on } \partial B_{1}. 
		\end{cases}
	\end{equation*} 
We have, if $\rho_0$ is small enough and $\omega_0$ is large enough, that   
	\begin{equation*} 
		\|(E \times \nu, H \times \nu)_{\inte}\|_{L^{2}(\partial B_1)}  \leq C \Big(  \|(h_1,h_2)\|_{L^2(\partial B_{1})} + \omega^{-1} \|(\dive_{\partial B_1}h_1, \dive_{\partial B_1}h_2)\|_{L^2(\partial B_{1})} \Big), 
	\end{equation*}
	for some positive constant $C$ depending only on $\omega_0$, $\omega_1$,  and $\rho_0$.
\end{lemma}

\begin{proof} 
	Applying Corollary \ref{cor-mor}, we have
	\begin{equation}\label{lem-lossy-e1}
	\|(E \times \nu|_{\inte}, H \times  \nu|_{\inte})\|_{L^2(\partial B_1)} \leq C\Big(\|(E, H)\|_{L^2(B_1\setminus B_{3/4})} + \|(E \cdot \nu, H \cdot  \nu)|_{\inte}\|_{L^2(\partial B_1)}\Big).
	\end{equation}
One has, see, e.g.,  \cite{Costabel}, 
	\begin{equation*}
	\|(E \cdot  \nu, H \cdot \nu) |_{\ext} \|_{L^2(\partial B_1)}  \le C \Big(  \|(E \times \nu, H \times  \nu) |_{\ext} \|_{L^2(\partial B_1)} + \| (E, H)\|_{L^2(B_2 \setminus B_1)} + \|(E, H) \|_{L^2(\partial B_2)} \Big). 
	\end{equation*}
Applying Lemma~\ref{lem:ex-123} for $(E, H)$ in $\R^3\setminus B_1$, we obtain
\begin{equation*}
 \| (E, H)\|_{L^2(B_2 \setminus B_1)} + \|(E, H) \|_{L^2(\partial B_2)} \leq C\|(E \times \nu, H \times  \nu) |_{\ext} \|_{L^2(\partial B_1)}.
\end{equation*} 
It follows that
\begin{equation}\label{lem-lossy-*}
	\|(E \cdot  \nu, H \cdot \nu) |_{\ext} \|_{L^2(\partial B_1)}  \le C \|(E \times \nu, H \times  \nu) |_{\ext} \|_{L^2(\partial B_1)}.
\end{equation}
Since
	\begin{equation*}
	\big(1  - (i \omega)^{-1} \big)E \cdot \nu |_{\inte} = \rho E \cdot \nu |_{\ext}  + \frac{1}{i \omega} \dive_{\partial B_1} h_2 \quad \mbox{ and } \quad H \cdot \nu |_{\inte} = \rho H \cdot \nu |_{\ext} - \frac{1}{i \omega} \dive_{\partial B_1} h_1, 
	\end{equation*}
	we derive from \eqref{lem-lossy-*} that 
	\begin{equation*}
	\|(E \cdot  \nu, H \cdot \nu) |_{\inte} \|_{L^2(\partial B_1)}  \le C \Big( \rho\|(E \times \nu, H \times  \nu) |_{\ext} \|_{L^2(\partial B_1)} + \omega^{-1}\|\dive_{\partial B_1}(h_1, h_2)\|_{L^2(\partial B_1)} \Big). 
	\end{equation*}
	From  the transmission conditions on $\partial B_1$, we deduce that 
	\begin{multline}\label{lem-lossy-e2}
	\|(E \cdot  \nu, H \cdot \nu) |_{\inte} \|_{L^2(\partial B_1)} \\[6pt] \le C \Big( \rho\|(E \times \nu, H \times  \nu) |_{\inte} \|_{L^2(\partial B_1)} + \rho\|(h_1, h_2)\|_{L^2(\partial B_1)} + 
	\omega^{-1}\|\dive_{\partial B_1}(h_1, h_2)\|_{L^2(\partial B_1)}
	 \Big). 
	\end{multline}
	
	On the other hand, as in \eqref{lem-lossy-low-e5.1}, we have 
	\begin{align}\label{lem-lossy-e2bis}
	\int_{B_1\setminus B_{1/2}}|E|^2\,dx  &\le   \left|\int_{\partial B_1} h_{2} \cdot \bar E |_{\ext} - \bar h_{1} H |_{\inte} \, ds\right|\\[6pt]
\nonumber	& \leq C\Big(\omega_0^2 \|(h_1, h_2)\|^2_{L^2(\partial B_1)} + \omega_0^{-2}\|(E\times\nu, H\times\nu)|_{\ext}\|^2_{L^2(\partial B_1)}\Big).
	\end{align}
	Since $\Delta E + \omega^2 E - i \omega E = 0$ in $B_1 \setminus B_{1/2}$, it follows that 
	\begin{equation}\label{lem-lossy-e3}
	\int_{B_{3/4}\setminus B_{2/3}}|E|^2 + \omega^{-2} |\nabla E|^2 \,dx     \leq C\Big(\omega_0^2 \|(h_1, h_2)\|^2_{L^2(\partial B_1)} + \omega_0^{-2}\|(E\times\nu, H\times\nu)|_{\ext}\|^2_{L^2(\partial B_1)}\Big).
	\end{equation}
	An integration by parts yields, for $2/3 < r < 3/4$, that 
	\begin{multline}\label{lem-lossy-e4}
	\omega^2\int_{B_1\setminus B_{r}}|H|^2 \, dx - \omega^2\int_{B_1\setminus B_{r}}|E|^2 \, dx \\
	= \Re \Big\{i\omega \int_{\partial B_1}\bar E|_{\inte} (H\times\nu|_{\inte})\,ds - i \omega \int_{\partial B_{r}}\bar E |_{\inte} (H \times\nu|_{\inte})\,ds\Big\}.
	\end{multline}
	Combining \eqref{lem-lossy-e2bis}, \eqref{lem-lossy-e3}, and \eqref{lem-lossy-e4} yields 
	\begin{equation}\label{lem-lossy-e5}
	\| (E, H) \|_{L^2(B_1 \setminus B_{3/4})} \le C \Big( \omega_0\|(h_1, h_2)\|_{L^2(\partial B_1)}  + \omega_0^{-1}\| (E \times \nu, H \times \nu) |_{\inte} \|_{L^2(\partial B_1)}\Big).
	\end{equation}	
	From \eqref{lem-lossy-e1}, \eqref{lem-lossy-e2}, and \eqref{lem-lossy-e5}, one obtains that, for $\rho$ small enough,
\begin{multline*}
\|(E \times \nu|_{\inte}, H \times  \nu|_{\inte})\|_{L^2(\partial B_1)}  \\[6pt]\leq C\Big(\omega_0\|(h_1, h_2)\|_{L^2(\partial B_1)} + \omega_0^{-1}\|(E\times\nu, H\times\nu)|_{\inte}\|_{L^2(\partial B_1)}  + \omega^{-1}\|\dive_{\partial B_1}(h_1, h_2)\|_{L^2(B_1)}\Big).
\end{multline*}
This implies 
\begin{equation*}
\|(E \times \nu|_{\inte}, H \times  \nu|_{\inte})\|_{L^2(\partial B_1)} \leq C\Big(\|(h_1, h_2)\|_{L^2(\partial B_1)} + \omega^{-1}\|\dive_{\partial B_1}(h_1, h_2)\|_{L^2(B_1)}\Big),
\end{equation*}
for $\omega_0$ large enough and $\rho$ small enough.
\end{proof}

\subsubsection{Proof of Proposition \ref{prop-av}.} 
Since $\omega>\omega_0$ is  large, by \eqref{STST}, one has
\begin{equation*}
\|\EEE(\rho .), \HHH(\rho .)\|_{L^2(\partial B_1)} + \omega^{-1}\| \dive_{\partial B_1} (\EEE(\rho .)\times\nu,   \dive_{\partial B_1} \HHH(\rho .)\times\nu)\|_{L^2(\partial B_1)} \leq C  \omega \|\JJJ\|_{L^2(\R^3)}.
\end{equation*}
Applying Lemma~\ref{lem-lossy}, we obtain 
\begin{equation*}
\|(\tbE_{\rho}, \tbH_{\rho})\|_{L^2(B_2 \setminus B_1)}  \le C \omega  \| \JJJ \|_{L^2(\R^3)}.
\end{equation*}
The conclusion now follows from Lemma~\ref{lem-FF-1}.  \qed

\subsection{High and very high frequency analysis - Proof of Proposition~\ref{pro-high}}\label{sc-high}

This section contains two subsections. In the first, we present several lemmas used in the proof of Proposition~\ref{pro-high} and in the second,   the proof of Proposition \ref{pro-high} is given.

\subsubsection{Some useful lemmas} 
We begin this section with  a trace-type result for Maxwell's equations in a bounded domain. The analysis is based on  a dual argument,  see, e.g.,  \cite{Lions, Cap2}). In this subsection, $D$ denotes a  smooth,   bounded, open subset of $\R^3$.

\begin{lemma}\label{lem-es} Let $\omega > \omega_0> 0$ and $f\in H(\dive, D)$. Assume that $(E, H) \in [H(\curl, D)]^2$ satisfies the equations
\begin{equation}\label{lem-es-e0}
\begin{cases}
\nabla\times E = i\omega H &\mbox{ in } D,\\[6pt]
\nabla\times H = -i\omega E + f &\mbox{ in } D.
\end{cases}
\end{equation}
Then
\[
\|E\|_{H^{-1/2}(\partial D)} +  \omega \| H \times \nu \|_{H^{-3/2}(\partial D)}  \leq C\Big( \omega^2\|E\|_{L^2(D)} + \omega\|f\|_{L^2(D)} + \omega^{-1}\| \dive f \|_{L^2(D)}\Big),
\]
for some positive constant $C$ depending only on $D$ and $\omega_0$. 
\end{lemma}

\begin{remark} \rm  It is crucial to our analysis that the constant $C$ is independent of $\omega$. 
\end{remark}

\begin{proof}  We have, from \eqref{lem-es-e0}, 
\begin{equation}\label{lem-es-Eq-E}
	\Delta E + \omega^2E = \nabla (\dive E) - \nabla \times (\nabla \times E) + \omega^2 E = \frac{1}{i \omega } \nabla( \dive f ) 
	- i \omega f \mbox{ in } D.
\end{equation}
Fix  $\phi \in [H^{1/2}(\partial D)]^3$ (arbitrary). By the trace theory, see, e.g.,  \cite[Theorem 1.6]{Girault}, there exists  $\xi\in [H^2(D)]^3$ such that
	\begin{equation}\label{lem-esbound-e-e1}
	\xi  = 0 \mbox{ on } \partial D, \quad
	 \frac{\partial \xi}{\partial \nu} = \phi \mbox{ on } \partial D, 
	\end{equation} 
	and
	\begin{equation}\label{lem-esbound-e-e2}
	\|\xi\|_{H^2(D)}\leq C\|\phi\|_{H^{1/2}(\partial D)}. 
	\end{equation}
Here and in what follows, $C$ denotes a positive constant depending only on $D$ and $\omega_0$.   Multiplying \eqref{lem-es-Eq-E} by $\xi$ and integrating by parts, we obtain 
	\begin{align}\label{lem-esbound-e-e3}
	\int_{D} (\Delta \xi + \omega^2 \xi) E - \int_{\partial D} E \phi =   \int_{D} (\Delta E + \omega^2 E) \xi =   \int_{D} - \frac{1}{i \omega} \dive f  \dive \xi  
	- i \omega f \xi.
		\end{align}
We derive from \eqref{lem-esbound-e-e1}, \eqref{lem-esbound-e-e2}, and \eqref{lem-esbound-e-e3} that 
\begin{equation*}
\left|\int_{\partial D}E\phi\, ds\right|\leq C\Big(\omega^2\|E\|_{L^2(D)} + \omega\|f\|_{L^2(D)}  + \omega^{-1}\| \dive f\|_{L^2(D)} \Big)\|\phi\|_{H^{1/2}(\partial D)}, 
\end{equation*}
which implies, since $\phi$ is arbitrary,  
\begin{equation}\label{lem-es-part1}
\|E \|_{H^{-1/2}(\partial D)} \le  C\Big(\omega^2\|E\|_{L^2(D)} + \omega\|f\|_{L^2(D)}  + \omega^{-1}\| \dive f\|_{L^2(D)} \Big).
\end{equation}

It remains to prove 
\begin{equation}\label{lem-es-claim2}
\|H\times\nu\|_{H^{-3/2}(\partial D)} \leq C\Big(\omega\|E\|_{L^2(D)} + \|f\|_{L^2(D)}+ \omega^{-2}\| \dive f \|_{L^2(D)}\Big).
\end{equation}
Fix $\varphi \in H^{3/2}(\partial D)$ (arbitrary), consider an extension of $\varphi$ in $D$ such that its $H^2(D)$-norm is bounded by  $C \| \varphi \|_{H^{3/2}(\partial D)}$,  and still denote this extension by $\varphi$. Such an extension exists by the trace theory, see, e.g., \cite[Theorem 1.6]{Girault}. We have 
\begin{equation}\label{cor-es-p0}
\int_{\partial D} H\times \nu\cdot  \varphi\, ds  = \int_{D}\Big(\nabla \times \varphi \cdot H - \nabla \times H \cdot \varphi\Big)\, dx . 
\end{equation}
Since
\begin{align*}
\left|\int_{D}\nabla\times \varphi \cdot H\, dx\right| & = \omega^{-1}\left|\int_{D}\nabla\times \varphi \cdot \nabla\times E\, dx\right|\\[6pt] 
\nonumber & =  \omega^{-1}\left|\int_{D}\nabla\times (\nabla\times \varphi) \cdot E\, dx + \int_{\partial D}E\cdot [( \nabla\times \varphi) \times\nu)] \, ds\right|, 
\end{align*}
and  $\nabla \times H = i \omega E + f$, 
it follows from \eqref{lem-es-part1}  that 
\begin{equation}\label{cor-es-p2}
\Big|\int_{D}\nabla\times \varphi \cdot H\, dx\Big| \leq C\Big(\omega\|E\|_{L^2(D)} + \|f\|_{L^2(D)} +  \omega^{-2}\| \dive f\|_{L^2(D)}\Big)\|\varphi\|_{H^{3/2}(\partial D)}
\end{equation}
and
\begin{equation}\label{cor-es-p1}
\Big|\int_{D}\nabla \times H \cdot \varphi\, dx\Big| \leq  C\Big(\omega\|E\|_{L^2(D)} + \|f\|_{L^2(D)}\Big)\|\varphi\|_{H^{3/2}(\partial D)}.
\end{equation}
Combining  \eqref{cor-es-p0},  \eqref{cor-es-p2}, and \eqref{cor-es-p1} yields 
\[
\left|\int_{\partial D} H\times \nu\cdot  \varphi\, ds\right| \leq C\Big(\omega\|E\|_{L^2(D)} + \|f\|_{L^2(D)} + \omega^{-2}\|\dive f\|_{L^2(D)}\Big)\|\varphi\|_{H^{3/2}(\partial D)}. 
\]
Since $\varphi$ is arbitrary, assertion~\eqref{lem-es-claim2} follows. The proof is complete.
\end{proof}

Using Lemma~\ref{lem-es}, we establish the following Lemma, which is the main result of this subsection.

\begin{lemma}\label{lem-high-due}
	Let $\omega > \omega_1> 0$, $0<\rho<1$, and  assume that $\omega\rho > \omega_1$.  Given $h_1, h_2 \in H^{3/2}(\dive_{\Gamma}, \partial B_1)$, let $(E, H) \in [\cap_{R> 1}H(\curl, B_R \setminus \partial B_1)]^2$ be the unique radiating solution of 
	\begin{equation*}
	\left\{\begin{array}{cl}
	\nabla\times E = i\omega\tilde \mu_{\rho}H &\mbox{ in } \R^3,\\[6pt]
	\nabla\times H = -i\omega\tilde \eps_\rho E + \tilde \sigma_{\rho} E &\mbox{ in } \R^3,\\[6pt]
	[E \times \nu] = h_1,  [H \times \nu] = h_2 &\mbox{ on } \partial B_1. 
	\end{array}\right.
	\end{equation*}
	We have
	\[
\|E\times\nu|_{\ext}\|_{H^{-1/2}(\partial B_1)} + \omega\|H \times\nu|_{\ext}\|_{H^{-3/2}(\partial B_1)} \leq C \Big( \omega^4 \|h_2\|_{H^{1/2}(\partial B_1)} + \omega^3 \|h_1\|_{H^{3/2}(\partial B_1)} \Big), \]
	for some positive constant $C$ depending only on $\omega_1$.
\end{lemma}

\begin{proof} As in \eqref{lem-lossy-low-e5.1}, we have 
 \begin{equation*}
	\int_{B_1\setminus B_{1/2}}|E|^2\,dx  \le   \left|\int_{\partial B_1} h_{2} \cdot \bar E |_{\ext} - \bar h_{1} H |_{\inte} \, ds\right|.
\end{equation*}
This implies 
 \begin{multline}\label{lem-high-due-e1}
	 \int_{B_1\setminus B_{1/2}}|E|^2\,dx  \le  \| h_2 \|_{H^{1/2}(\partial B_1)}     \| E |_{\inte}\|_{H^{-1/2}(\partial B_1)} \\[6pt]
	 +  \| h_1 \|_{H^{3/2}(\partial B_1)}     \| H \times \nu |_{\inte}\|_{H^{-3/2}(\partial B_1)} + \| h_2\|_{L^2(\partial B_1)}^2.
\end{multline}
Applying  Lemma \ref{lem-es} to $(E, H)$ with  $f = E$ in $B_1\setminus B_{1/2}$,  we have
	\[
	 \| E |_{\inte}\|_{H^{-1/2}(\partial B_1)}  + \omega \|H \times\nu\|_{H^{-3/2}(\partial B_{1})} \leq C\omega^2 \|E\|_{L^2(B_1\setminus B_{1/2})}.
	\]
It follows  from \eqref{lem-high-due-e1} that
	\begin{equation*}
	\|E\|_{L^2(B_1\setminus B_{1/2})}\leq C \Big( \omega^2 \|h_2\|_{H^{1/2}(\partial B_1)} + \omega \|h_1\|_{H^{3/2}(\partial B_1)} \Big).
	\end{equation*}
Applying   Lemma \ref{lem-es}  to $(E, H)$ with  $f = E$ in $B_1\setminus B_{1/2}$ again, one has 
\[
\|E\times\nu|_{\inte}\|_{H^{-1/2}(\partial B_1)} + \omega\|H\times\nu|_{\inte}\|_{H^{-3/2}(\partial B_1)} \leq C \Big( \omega^4 \|h_2\|_{H^{1/2}(\partial B_1)} + \omega^3 \|h_1\|_{H^{3/2}(\partial B_1)} \Big).\]
Using the transmission condition at $\partial B_1$, one reaches the conclusion. 
\end{proof}

We end this subsection by a simple consequence of Stratton-Chu's formula. 

\begin{lemma}\label{lem-FF}Let $0 < \rho < 1$,  $\omega  > \omega_1 > 0$ be such that   $\omega \rho > \omega_1$, and let $D \subset B_1$. Assume that  $(E, H) \in \big[H_{\loc}(\curl, \mR^3 \setminus D) \big]^2$ is a  radiating solution to the Maxwell equations
	\[
	\left\{\begin{array}{cl}
	\nabla\times E = i \omega \rho H & \mbox{ in }  \R^3\setminus \bar{D}, \\[6pt] 
	\nabla\times H = -i \omega \rho E & \mbox{ in } \R^3\setminus \bar{D}.
	\end{array}\right.
	\]
	We have
	\begin{equation*}
	|E(x)| \le \frac{C |\omega \rho|^{3/2}}{|x|}  \|E \times \nu \|_{H^{-1/2}(\partial D)}  +   \frac{C |\omega \rho|^{5/2}}{|x|}   \| H \times \nu\|_{H^{-3/2}(\partial D)} \mbox{ for } x \in B_{3/ \rho} \setminus B_{1/ \rho}, 
	\end{equation*}
for some positive constant $C$ independent of $x$, $\omega$, and $\rho$. 
\end{lemma}

\subsubsection{Proof of Proposition \ref{pro-high}.} Apply Lemma \ref{lem-high-due}, we have
\begin{multline}\label{pro-high-e1}
\|\tbE_{\rho}  \times \nu \|_{H^{-1/2}(B_{2}\setminus B_1)} + \omega  \|\tbH_{\rho} \times \nu \|_{H^{-3/2}(B_{2}\setminus B_1)}  \\[6pt]
\leq  C \omega^3 \| \EEE(\rho\, \cdot) \times \nu \|_{H^{3/2}(\partial B_1)} + C \omega^4 \| \HHH(\rho\, \cdot) \times \nu \|_{H^{1/2}(\partial B_1)}.
\end{multline}
Since $\omega>\omega_0$, which is  large, by \eqref{STST}, one has
\begin{equation}\label{pro-high-e2}
 \omega^3 \| \EEE(\rho\, \cdot) \times \nu \|_{H^{3/2}(\partial B_1)} +  \omega^4 \| \HHH(\rho\, \cdot) \times \nu \|_{H^{1/2}(\partial B_1)} \leq C  \omega^{6} \rho^{1/2} \|\JJJ\|_{L^2(\R^3)}.
\end{equation}
Applying Lemma~\ref{lem-FF}, we derive from \eqref{pro-high-e1} and \eqref{pro-high-e2} that
\begin{equation*}
\|\tbE_{\rho}\|_{L^2(B_3 \setminus B_{1/2})}  \le C \omega^{15/2} \rho^{3}  \| \JJJ \|_{L^2(\R^3)}, 
\end{equation*}
which yields 
\begin{equation*}
\|\tbH_{\rho}\|_{L^2(B_2 \setminus B_1)}  \le C \omega^{17/2} \rho^{3}  \| \JJJ \|_{L^2(\R^3)}. 
\end{equation*}
The proof is complete. 
\qed

\section{Proof of Theorem \ref{thm-main-1}}\label{sc-proof-main}

To implement the analysis in the frequency domain, let us introduce the notation for the Fourier transform with respect to $t$: 
\begin{equation}\label{def-F}
\hat u(\omega, x) = \frac{1}{\sqrt{2 \pi}} \int_{\mR} u(t, x) e^{i \omega t} \, dt,  
\end{equation}
for an appropriate function $u \in L^\infty_{\loc}([0, + \infty), L^2(\mR^3))$; here we extend $u$ by $0$ for $t < 0$. 

\medskip 
The starting point of the frequency analysis is based on  the following result:

\begin{proposition}\label{prop-radiating}  Let 
$f_e, f_m \in L^2 \big([0, \infty); [L^2(\R^3)]^3 \big)\cap L^1 \big([0, \infty); [L^2(\R^3)]^3 \big)$. Let \\$(\EE, \HH)$
			$\in L^{\infty}_{\loc} \big([0, +\infty), [L^2(\R^3)]^6 \big)$ be the unique weak solution of \eqref{equ}.  
Assume that there exists $R_0 > 0$ such that $\supp f_e(t, \cdot), \,  \supp f_m(t, \cdot), \,  \supp \sigma_e, \,  \supp \sigma_m \subset B_{R_0}$ for $t > 0$. 
						Then,  for almost every $\omega > 0$, $(\hat \EE, \hat \HH)(\omega, .) \in  [H_{\loc}(\curl, \R^3)]^2$ is the unique, {\bf radiating} solution to the system
			\begin{equation}\label{eq-freq-0}
			\begin{cases}
			\nabla\times \hat \EE(\omega, .) = i\omega\mu \hat\HH(\omega, .) - \sigma_m \hat \HH(\omega, \cdot) + \hat f_e(\omega, \cdot) &\text{ in } \R^3,\\[6pt]
			\nabla\times \hat \HH(\omega, .) = -i\omega\eps \hat \EE(\omega, .)+ \sigma_e \hat \EE(\omega, .) - \hat f_m (\omega, .) &\text{ in }  \R^3.
			\end{cases}
			\end{equation}
		\end{proposition}
		
\begin{proof}
Let $(\EE_{\delta}, \HH_{\delta})\in L_{\loc}^{\infty} \big([0, \infty), [L^2(\R^3)]^6 \big)$ be the unique weak solution to 
\begin{equation*}
\begin{cases}
\dsp  \eps \frac{\partial \EE_{\delta}}{\partial t} = \nabla\times \HH_{\delta} -  \sigma_e\EE_{\delta} - \delta \EE_{\delta} + f_m &\text{ in } (0, +\infty)\times \R^3,\\[6pt]
\dsp\mu \frac{\partial \HH_{\delta}}{\partial t} = -  \nabla\times \EE_{\delta}  -\sigma_m\HH_{\delta} - \delta \HH_{\delta}  + f_e&\text{ in } (0, +\infty)\times \R^3,\\[6pt]
\dsp \EE_{\delta}(0, ) = 0; \HH_{\delta}(0, ) = 0 & \text{ in } \R^3. 
\end{cases}
\end{equation*}
By the  standard Galerkin approach (see e.g., \cite{HM4}), one can prove that 
	\begin{equation*}
	\delta \int_{0}^{+\infty} \int_{\R^3}|\EE_{\delta}(s, x)|^2+ |\HH_{\delta}(s, x)|^2\, dx\, ds \leq C \| (f_e, f_m) \|^2_{L^2(\mR_+, \mR^3)},
	\end{equation*}
	for some positive constant independent of $\delta$ and $(f_e, f_m)$. Hence  $\EE_{\delta}, \HH_{\delta}\in L^2 \big((0, \infty); [L^2(\R^3)]^3 \big)$, and thus $\hat \EE_{\delta}, \hat \HH_{\delta}\in L^2 \big((0, \infty); [L^2(\R^3)]^3 \big)$ by Parserval's theorem. It follows, for a.e.  $\omega > 0$, that $(\hat \EE_{\delta}, \hat \HH_{\delta})\in H(\curl, \R^3)$ is the unique solution to
\begin{equation}\label{eq-freq}
\begin{cases}
\nabla\times \hat \EE_{\delta}(\omega, .) = i\omega\mu \hat\HH_{\delta}(\omega, .) - (\sigma_m + \delta)\hat \HH_{\delta}(\omega, .) + \hat f_e(\omega, \cdot)  &\text{ in } \R^3,\\[6pt]
\nabla\times \hat \HH_{\delta}(\omega, .) = -i\omega\eps \hat \EE_{\delta}(\omega, .)+ (\sigma_e+ \delta) \hat \EE_{\delta}(\omega, .) - \hat f_m (\omega, .)  &\text{ in }  \R^3.
\end{cases}
\end{equation}
For $0<\omega_1 < \omega < \omega_2 < \infty$, one can check that the solution of \eqref{eq-freq} satisfies
\begin{equation}\label{es-freq-1}
\|(\hat \EE_{\delta}, \hat \HH_{\delta})(\omega, .)\|_{H(\curl, B_R)}\leq C \|(\hat f_e, \hat f_m)(\omega, .)\|_{L^2(\R^3)}  \le C  \|(f_e, f_m)\|_{L^1((0, \infty), L^2(\R^3))}, 
\end{equation}
for some positive constant $C$ depending only on $\eps,	 \mu$, $R$, $\omega_1$,  and $\omega_2$.  Letting $\delta\to 0$ and using the limiting absorption principle, see e.g.,  \cite[(2.28) and the following paragraph]{HM1},  one derives that 
\begin{equation}\label{es-freq-2}
(\hat \EE_{\delta}, \hat \HH_{\delta})(\omega, )\rightharpoonup (\EE_0, \HH_0)(\omega, .) \mbox{ weakly in } [H_{\loc}(\curl, \R^3)]^2 \mbox{ as } \delta \to 0, 
\end{equation}
where $(\EE_0, \HH_0)(\omega, .)\in [H_{\loc}(\curl, \R^3)]^2$ is the unique, radiating solution to the system
\begin{equation*}
			\begin{cases}
			\nabla\times \EE_0(\omega, .) = i\omega\mu \HH_0(\omega, .) - \sigma_m \HH_0 + \hat f_e(\omega, \cdot) &\text{ in } \R^3,\\[6pt]
			\nabla\times \HH_0(\omega, .) = -i\omega\eps \EE_0(\omega, .) + \sigma_e \EE_0(\omega, .) - \hat{f}_m(\omega, .) &\text{ in }  \R^3.
			\end{cases}
\end{equation*}
From \eqref{es-freq-1} and \eqref{es-freq-2}, we have
\begin{equation}\label{prop-1-p1}
(\hat \EE_{\delta}, \hat \HH_{\delta}) \to (\EE_0, \HH_0) \mbox{ in the distributional sense in }  \R_+\times \R^3 \mbox{ as } \delta \to 0.
\end{equation}
We claim that 
\begin{equation}\label{prop-1-p2}
(\hat \EE_{\delta}, \hat \HH_{\delta}) \to (\hat \EE, \hat \HH) \mbox{ in the distributional sense in }  \R_+\times \R^3, 
\end{equation}
and the conclusion follows from \eqref{prop-1-p1} and \eqref{prop-1-p2}. 

It remains to prove \eqref{prop-1-p2}. Let $\phi\in [C^{\infty}_{c} \big((0, \infty)\times \R^3)\big]^3$. We have
\begin{equation*}
\int_{\R}\int_{\R^3}(\hat \EE_{\delta}(\omega, x) - \hat \EE(\omega, x))\bar \phi(\omega, x)\, dx d\omega = \int_{\R}\int_{\R^3}(\EE_{\delta}(t, x) -  \EE(t, x))\bar{\check \phi}(t, x)\, dx dt.
\end{equation*}
We derive that, by applying Proposition \ref{prop-wellposed} to $(\EE_{\delta} - \EE, \HH_{\delta} - \HH)$, 
\begin{equation*}
\|\EE_{\delta}(t, .) - \EE(t, .)\|_{L^2(\R^3)} \leq C \delta \int_{0}^t\|(\EE(s, .), \HH(s, .))\|_{L^2(\R^3)}\, ds \mbox{ for } t > 0, 
\end{equation*}
and,  by applying Proposition \ref{prop-wellposed} for $(\EE, \HH)$, 
\[
\|(\EE(t, .), \HH(t, .))\|_{L^2(\R^3)} \leq C\|(f_e, f_m)\|_{L^1((0, \infty), [L^2(\R^3)]^6)} \mbox{ for }  t> 0.
\]
It follows that  
\begin{equation}\label{prop-radiating-e1}
\|\EE_{\delta}(t, .) - \EE(t, .)\|_{L^2(\R^3)} \leq C\delta t. 
\end{equation}
From \eqref{prop-radiating-e1}, we obtain
\begin{equation}\label{prop-radiating-e2}
\int_{\R}\int_{\R^3}(\EE_{\delta}(t, x) -  \EE(t, x))\bar{\check \phi}(t, x)\, dx dt\leq C\delta\int_{\R}t\|\check\phi(t, .)\|_{L^2(\R^3)}\, dt.
\end{equation}
From \eqref{prop-radiating-e2} and the fast decay property of $\check \phi$, we derive that  
\[
\hat \EE_{\delta} \to \hat \EE \mbox{ in the distributional sense in }  \R_+\times \R^3.
\]
Similarly, one can prove that 
\[
\hat \HH_{\delta} \to \hat \HH \mbox{ in the distributional sense in }  \R_+\times \R^3.
\]
The proof is complete.
\end{proof}

We are ready to give 

\medskip 

\noindent{\bf Proof of Theorem~\ref{thm-main-1}.}   Fix $K \subset\subset \R^3\setminus \bar B_1$ and $T> 0$.  Using the fact that $\hat \EE_c(-k, x) =\overline{ \hat \EE_c}(k, x)$ and $\hat \EE(-k, x) =\overline{ \hat \EE}(k, x)$ for $k>0$, one has, for $0< t< T$,
\begin{equation}\label{thm-e3-3}
\|\EE_c(t, \cdot) - \EE(t, \cdot)\|_{L^2(K)}  \le \int_{0}^T \| \partial_t \EE_c(t, \cdot) - \partial_t\EE(t, \cdot)\|_{L^2(K)}   \le T \int_{0}^{\infty} \omega \|\hat \EE_c (\omega, \cdot)  - \hat\EE (\omega, \cdot) \|_{L^2(K)}d\omega.
\end{equation}
We have, by  Proposition \ref{pro-low},  
	\begin{equation}\label{thm-e0.2}
	 \int\limits_{0}^{1} \omega \|\hat \EE_c(\omega, .) - \hat \EE(\omega, .)\|_{L^2(K)} d\omega \leq C\int\limits_{0}^{1}\rho^3\| \hat{\mathcal J}(\omega, .)\|_{L^2(\R^3)}d\omega  \le  C\rho^3\|\mathcal J\|^2_{L^2(\R; L^2(\R^3))}, 
	\end{equation}
by  Proposition \ref{prop-av} (here to  simplify the notations we assume that $\omega_0 = 1$), 
	 \begin{equation}\label{thm-e0.3}
	 \int\limits_{1}^{1/\rho} \omega \|\hat \EE_c(\omega, .) - \hat \EE(\omega, .)\|_{L^2(K)} d\omega \leq C\rho^3 \int\limits_{1}^{1/ \rho} \omega^4  \|\hat{\mathcal J}(\omega, .)\|_{L^2(\R^3)}d\omega, 
	 \end{equation}
and,  by Proposition~\ref{pro-high}, 
 \begin{equation}\label{thm-e0.4}
	 \int\limits_{1/ \rho}^{+\infty} \omega \|\hat \EE_c(\omega, .) - \hat \EE(\omega, .)\|_{L^2(K)} d\omega \leq C\rho^3 \int\limits_{\frac{1}{\rho}}^{+\infty} \omega^{19/2}  \|\hat{\mathcal J}(\omega, .)\|_{L^2(\R^3)}d\omega. 
	 \end{equation}
A combination of \eqref{thm-e0.3} and \eqref{thm-e0.4} yields 
 \begin{align}\label{thm-e0.6}
 \int\limits_{1}^{\infty} \omega \|\hat \EE_c(\omega, .) - \hat \EE(\omega, .)\|_{L^2(K)} d\omega &\le C\rho^3 \int_1^{+\infty}\frac{1}{\omega}\|\widehat{\partial_t^{(11)}\mathcal J}(\omega, \cdot)\|_{L^2(\R^3))}\, d\omega\\
  \nonumber&\leq C\rho^3\|\mathcal J\|_{H^{11}(\R, L^2(\R^3))}
 \end{align}
We derive from  \eqref{thm-e3-3}, \eqref{thm-e0.2}, and \eqref{thm-e0.6} that, for $0< t< T$, 
\begin{equation*}
\|\EE_c(t, \cdot) - \EE(t, \cdot)\|_{L^2(K)}  \leq C T \rho^3\|\mathcal J\|_{H^{11}(\R, L^2(\R^3))}.
\end{equation*}
The proof is complete. \qed

\end{document}